\documentclass[final,twoside,11pt]{entics} 
\usepackage{enticsmacro}
\usepackage{enumerate,xspace}
\usepackage{amsmath,amssymb,wasysym}
\usepackage[all]{xy}
\usepackage{proof}
\usepackage[svgnames]{xcolor}
\usepackage{tikz}
\usepackage[sort,nocompress]{cite}
\usepackage{stmaryrd} 
\usepackage{mathtools}
\usepackage{latexsym}
\usepackage{dsfont}
\usepackage{multicol}
\usepackage{cmll}
\usepackage{multirow}
\usepackage{longtable}
\allowdisplaybreaks

\def\conf{MFPS 2024}
\def\myconf{mfps40} 	
\volume{4}			
\def\volu{4}			
\def\papno{15}			
\def\mydoi{14670}%
\def\lastname{Jean-Simon Pacaud Lemay} 
\def\runtitle{An Ultrametric for CDC for Taylor Series Convergence} 
\def\copynames{J.-S. P. Lemay}    
\begin{document}
\begin{frontmatter}
  \title{An Ultrametric for Cartesian Differential Categories for\\[1ex] Taylor Series Convergence\thanksref{ALL}} 						
 \thanks[ALL]{Thank you to the organizers of QUALOG2023 for inviting me to give a talk on this research, and to the audience of QUALOG2023 for listening to the talk and their questions. The author is financially supported by an ARC DECRA (DE230100303) and an AFOSR Research Grant (FA9550-24-1-0008). This material is based upon work supported by the Air Force Office of Scientific Research under award number FA955024-1-0008.}   
  \author{Jean-Simon Pacaud Lemay\thanksref{a}\thanksref{myemail}}	
   \address[a]{Macquarie Univeristy\\				
   Sydney, NSW, Australia}  							
   \thanks[myemail]{Email: \href{mailto:js.lemay@mq.edu.au}{\texttt{\normalshape
      js.lemay@mq.edu.au}} }
\begin{abstract} 
Cartesian differential categories provide a categorical framework for multivariable differential calculus and also the categorical semantics of the differential $\lambda$-calculus. Taylor series expansion is an important concept for both differential calculus and the differential $\lambda$-calculus. In differential calculus, a function is equal to its Taylor series if its sequence of Taylor polynomials converges to the function in the analytic sense. On the other hand, for the differential $\lambda$-calculus, one works in a setting with an appropriate notion of algebraic infinite sums to formalize Taylor series expansion. In this paper, we provide a formal theory of Taylor series in an arbitrary Cartesian differential category without the need for converging limits or infinite sums. We begin by developing the notion of Taylor polynomials of maps in a Cartesian differential category and then show how comparing Taylor polynomials of maps induces an ultrapseudometric on the homsets. We say that a Cartesian differential category is Taylor if maps are entirely determined by their Taylor polynomials. The main results of this paper are that in a Taylor Cartesian differential category, the induced ultrapseudometrics are ultrametrics and that for every map $f$, its Taylor series converges to $f$ with respect to this ultrametric. This framework recaptures both Taylor series expansion in differential calculus via analytic methods and in categorical models of the differential $\lambda$-calculus (or Differential Linear Logic) via infinite sums.  
\end{abstract}
\begin{keyword}
Cartesian Differential Categories; Taylor Series; Taylor Polynomials 
\end{keyword}
\end{frontmatter}
\section{Introduction}\label{intro}

Cartesian differential categories \cite{blute2009cartesian} provide the categorical foundations of multivariable differential calculus. The key feature of a Cartesian differential category is that it comes equipped with a \emph{differential combinator} $\mathsf{D}$, which for every map $f: A \to B$ produces its derivative $\mathsf{D}[f]: A \times A \to B$, generalizing the concept of the total derivative from differential calculus. Cartesian \emph{closed} differential categories provide the categorical semantics of the differential $\lambda$-calculus and the resource calculus \cite{bucciarelli2010categorical,Cockett-2019,ehrhard2003differential,manzonetto2012categorical,cockett2016categorical}. Cartesian differential categories are also closely related to differential categories \cite{blute2006differential,Blute2019}, which provide the categorical semantics of Differential Linear Logic \cite{ehrhard2006differential,ehrhard2017introduction}. Explicitly, the coKleisli category of a differential category is a Cartesian differential category \cite{blute2009cartesian,garner2020cartesian}. Cartesian differential categories have also been successful in formalizing various important concepts in differential calculus, such as solving differential equations and exponential functions \cite{lemay2020exponential}, Jacobians and gradients \cite{EPTCS372.3}, de Rham cohomology \cite{cruttwell2013forms}, linearization \cite{cockett2020linearizing}, etc. Cartesian (closed) differential categories have also been quite popular in computer science, in particular since they provide categorical frameworks for differential programming languages \cite{reversesemantics} and automatic differentiation \cite{catML}.

An all-important concept in differential calculus is the notion of \emph{Taylor series}. Recall that for a smooth function $f: \mathbb{R} \to \mathbb{R}$, its Taylor series at $0$ (sometimes also called its Maclaurin series) is the power series:
\[\mathcal{T}(f)(x) = \sum^\infty_{k=0} \frac{1}{k!}\cdot f^{(k)}(0)x^k\]
where $f^{(k)}$ is the $k$-th derivative of $f$. Famously, Taylor series are very useful, and a lot of information about a function can be gained from studying its Taylor series. It is often highly desirable for Taylor series to converge (in the usual real analytical sense) and also for functions to equal their Taylor series. 

The concept of Taylor series expansion is also important for both the differential $\lambda$-calculus and the resource calculus \cite{ehrhard2008uniformity,boudes2013characterization,ehrhard2006bohm,manzonetto2011bohm}. In the presence of countable infinite sums, one can give full Taylor expansions of $\lambda$-terms. In particular, Taylor expansion provides a linear approximation of ordinary application, so:
\[MN = \sum^\infty_{n=0} \frac{1}{n!} (\mathsf{D}^n M \cdot N^n) 0\] 
where $\mathsf{D}^n M \cdot N^n$ denotes the $n$-th derivative of $M$ applied $n$ times to $N$. In \cite{manzonetto2012categorical}, Manzonetto provides the categorical interpretation of this version of Taylor expansion in a Cartesian closed differential category with infinite sums. 

Taylor series expansion is also considered in Differential Linear Logic \cite{pagani2009inverse,boudes2013characterization,ehrhard2008uniformity}. In \cite{ehrhard2017introduction}, Ehrhard provides the categorical versions of these ideas in differential categories with countable sums. Ehrhard's approach to Taylor expansion in a differential category and Manzonetto's approach to Taylor expansion in a Cartesian closed differential category are the same when we consider the Taylor expansion of coKleisli maps (which is expected since recall that the coKleisli category of a differential category is a Cartesian differential category). Moreover, Taylor series expansion of coKleisli maps in a differential category played a fundamental role in developing \emph{codigging} for differential categories and Differential Linear Logic \cite{kerjean2023taylor}. Taylor series expansion was also considered in the coherent differentiation setting in \cite{ehrhard2023coherent} (though this framework, while possibly related, is quite different from Cartesian differential categories and the story of this paper). 

While the notion of Taylor series for the differential $\lambda$-calculus and Differential Linear Logic is fundamentally the same as the notion of Taylor series in classical differential calculus, there is a crucial difference. Indeed, to define Taylor series expansion for the differential $\lambda$-calculus or Differential Linear Logic, one first assumes that we are in a setting with countable infinite sums in the algebraic sense. While this is a perfectly fine thing to do in models coming from computer science, having infinite sums clashes with models coming from the analysis side of things. There are many important examples of Cartesian differential categories that do not have infinite sums and yet still have a well-defined notion of Taylor series, such as real smooth/entire functions, polynomials, etc. The objective of this paper is to provide a formal theory of Taylor series in an arbitrary Cartesian differential category in such a way that it also gives a unified story recapturing Taylor series expansion in differential calculus via convergence and for the differential $\lambda$-calculus via algebraic infinite sums.  

The key operation for developing Taylor series in Cartesian differential categories is the notion of \emph{higher-order derivatives}, which we review in Sec \ref{sec:CDC}. Then, in a setting where we can scalar multiply by positive rationals, we use higher-order derivatives to define the \emph{Taylor monomials} of maps in a Cartesian differential category. The sum of these Taylor monomials gives the \emph{Taylor polynomials} of maps. In Sec \ref{sec:TPOLY}, we develop the theory of Taylor polynomials in a Cartesian differential category and show that the Taylor polynomials form a sub-Cartesian differential category. In Sec \ref{sec:DPOLY}, we also discuss the intermediate notion of \emph{differential polynomials}, which are the maps in a Cartesian differential category whose higher-order derivative is eventual zero. We will see that every Taylor polynomial is a differential polynomial, but also explain why the converse is not necessarily true. 

In Sec \ref{sec:ultra}, we define an ultrapseudometric on the homsets of a Cartesian differential category where the distance between two maps is given by comparing their Taylor monomials. This ultrapseudometric is an analogue of the metric for power series that makes the formal infinite sum of a power series converge. However, in an arbitrary Cartesian differential category, two maps can have the same Taylor monomials but not be equal. We say that a Cartesian differential category is \emph{Taylor} if maps are completely determined by their Taylor monomials. The main result is that in a Taylor Cartesian differential category, the ultrapseudometric is an ultrametric, which then implies that the sequence of Taylor polynomials of any map converges to said map. In other words, in a Taylor Cartesian differential category, every map is equal to its Taylor series. We also explain how we obtain a Taylor Cartesian differential category from a Cartesian differential category by taking equivalence classes on maps with the same Taylor monomials. As running examples throughout the paper, we consider polynomials over a commutative semiring, real smooth functions, and the coKleisli category of a differential category. In particular, we will also see how for the coKleisli category of a differential category, this story of ultrametrics and Taylor series from the Cartesian differential category point of view corresponds precisely to the same named notions used in \cite{kerjean2023taylor} to construct codigging and exponential functions in differential categories. 

We also explain why if maps were equal to their Taylor series via some other notion of convergence or algebraic infinite sums, then the Cartesian differential category is indeed Taylor. In particular, in Sec \ref{sec:infinite}, we show how being Taylor recaptures precisely Manzonetto's notion of modelling Taylor expansion in a Cartesian closed differential category from \cite{manzonetto2012categorical}. 

\section{Cartesian Differential Categories}\label{sec:CDC}

In this background section, we review the basics of Cartesian differential categories, primarily to set up terminology and notation. We refer the reader to \cite{blute2009cartesian,cockett2020linearizing,garner2020cartesian} for a more in-depth introduction to Cartesian differential categories. 

In an arbitrary category $\mathbb{X}$, objects will be denoted by capital letters $A,B,X,Y$, etc. and maps by minuscule letters $f,g,h$, etc. We denote homsets by $\mathbb{X}(A,B)$, maps as arrows $f: A \to B$, identity maps as $1_A: A \to A$, and we use the classical notation for composition, $\circ$, as opposed to diagrammatic order which was used in other papers on Cartesian (reverse) differential categories, such as in \cite{blute2009cartesian}. For a category with finite products, we denote the product by $\times$, the projection maps by ${\pi_j: A_0 \times \hdots \times A_n \to A_j}$, and the pairing operation as $\langle -, \hdots, - \rangle$.

In this paper, we will work with Cartesian differential categories relative to a fixed commutative semiring $k$, as was done in \cite{garner2020cartesian}. As such, in this relative setting, the underlying structure of a Cartesian differential category is that of a \emph{Cartesian left $k$-linear category} \cite[Sec 2.1]{garner2020cartesian}, which can be described as a category with finite products which is \emph{skew}-enriched over the category of $k$-modules and $k$-linear maps between them \cite{garner2020cartesian}. Essentially, this means that each homset is a $k$-module, so we have zero maps, can take the sum of maps and also scalar multiply maps by elements of $k$, but also allow for maps which do not preserve this $k$-module structure. Maps which do are called \emph{$k$-linear maps}. Explicitly, a \textbf{left $k$-linear category} is a category $\mathbb{X}$ such that each homset $\mathbb{X}(A,B)$ is a $k$-module with scalar multiplication $\cdot : k \times  \mathbb{X}(A,B) \to  \mathbb{X}(A,B)$, addition ${+: \mathbb{X}(A,B) \times \mathbb{X}(A,B) \to \mathbb{X}(A,B)}$, and zero $0 \in \mathbb{X}(A,B)$; and such that pre-composition preserves the $k$-linear structure: $(r \cdot f + s \cdot g) \circ x = r \cdot (f \circ x) + s \cdot (g \circ x)$. A map $f: A\to B$ is said to be \textbf{$k$-linear} if post-composition by $f$ preserves the $k$-linear structure: $f \circ (r \cdot x + s \cdot y) =   r \cdot (f \circ x) + s \cdot (f \circ y)$. Then a \textbf{Cartesian left $k$-linear category} is a left $k$-linear category $\mathbb{X}$ such that $\mathbb{X}$ has finite products and all projection maps $\pi_j$ are $k$-linear. We note that when taking $k=\mathbb{N}$, the semiring of natural numbers, (Cartesian) left $\mathbb{N}$-linear categories and their $\mathbb{N}$-linear maps are the same thing as (Cartesian) left additive categories and their additive maps from \cite[Def 1.1.1 \& 1.2.1]{blute2009cartesian}. 

A Cartesian differential category is a Cartesian $k$-linear category that comes equipped with a \emph{differential combinator}, an operator that sends maps to their derivative. When taking $k= \mathbb{N}$, a Cartesian $\mathbb{N}$-differential category is precisely the same thing as the original definition from \cite[Def 2.1.1]{blute2009cartesian}. It is important to note that in this paper, we follow the now more widely used convention for Cartesian differential categories, which flips the convention used in early works such as in \cite{blute2009cartesian,manzonetto2012categorical}, so that the linear argument of the derivative is in the second argument rather than in the first. 

\begin{definition}\label{cartdiffdef} A \textbf{Cartesian $k$-differential category} \cite[Sec 2.2]{garner2020cartesian} is a Cartesian left $k$-linear category $\mathbb{X}$ equipped with a \textbf{differential combinator} $\mathsf{D}$, which is a family of functions (indexed by pairs of objects of $\mathbb{X}$) $\mathsf{D}: \mathbb{X}(A,B) \to \mathbb{X}(A \times A,B)$ such that the seven axioms \textbf{[CD.1]} to \textbf{[CD.7]} described below hold. For a map $f: A \to B$, the map $\mathsf{D}[f]: A \times A \to B$ is called the \textbf{derivative} of $f$. 
\end{definition}

The seven axioms of a differential combinator are analogues of the basic properties of the total derivative from multivariable differential calculus. They are that: \textbf{[CD.1]} the differential combinator is a $k$-linear morphism, \textbf{[CD.2]} derivatives are $k$-linear in their second argument, \textbf{[CD.3]} identity maps and projections are \emph{differential} linear, \textbf{[CD.4]} the derivative of a pairing is the pairing of the derivatives, \textbf{[CD.5]} the chain rule for the derivative of a composition, \textbf{[CD.6]} derivatives are \emph{differential} linear in their second argument, and lastly \textbf{[CD.7]} the symmetry of the mixed \emph{partial} derivatives. We will recall the notions of differential linearity and partial derivatives below. Cartesian differential categories have a very practical term calculus \cite[Sec 4]{blute2009cartesian}, which is highly useful for writing definitions and proofs. So, we write the derivative as follows:
\[\mathsf{D}[f](a,b) := \dfrac{\mathsf{d}f(x)}{\mathsf{d}x}(a) \cdot b\]
Then the axioms of the differential combinator are that: 
 \begin{enumerate}[{\bf [CD.1]}]
 \item $\dfrac{\mathsf{d} r \cdot f(x) + s \cdot g(x)}{\mathsf{d}x}(a) \cdot b = r \cdot \dfrac{\mathsf{d}f(x)}{\mathsf{d}x}(a) \cdot b + s \cdot \dfrac{\mathsf{d}g(x)}{\mathsf{d}x}(a) \cdot b$;
 \item $\dfrac{\mathsf{d}f(x)}{\mathsf{d}x}(a) \cdot (r\cdot b + s \cdot c) = r\cdot \dfrac{\mathsf{d}f(x)}{\mathsf{d}x}(a) \cdot b + s \cdot \dfrac{\mathsf{d}f(x)}{\mathsf{d}x}(a) \cdot b$
 \item $ \dfrac{\mathsf{d}x_j}{\mathsf{d}x_j}(a_0, \hdots, a_n) \cdot (b_0, \hdots, b_n) = b_j$ 
 \item $\dfrac{\mathsf{d}\left \langle f_0(x), \hdots, f_n(x) \right \rangle}{\mathsf{d}x}(a) \cdot b = \left \langle \dfrac{\mathsf{d}f_0(x)}{\mathsf{d}x}(a) \cdot b, \hdots, \dfrac{\mathsf{d}f_n(x)}{\mathsf{d}x}(a) \cdot b \right \rangle$
 \item $\dfrac{\mathsf{d}g\left(f(x) \right)}{\mathsf{d}x}(a) \cdot b = \dfrac{\mathsf{d}g(y)}{\mathsf{d}y}(f(a)) \cdot \left( \dfrac{\mathsf{d}f(x)}{\mathsf{d}x}(a) \cdot b \right)$
 \item $\dfrac{\mathsf{d}\dfrac{\mathsf{d}f(x)}{\mathsf{d}x}(y) \cdot z}{\mathsf{d}(y,z)}(a,0) \cdot (0,b) = \dfrac{\mathsf{d}f(x)}{\mathsf{d}x}(a) \cdot b$
 \item $\dfrac{\mathsf{d}\dfrac{\mathsf{d}f(x)}{\mathsf{d}x}(y) \cdot z}{\mathsf{d}(y,z)}(a,b) \cdot (c,d) = \dfrac{\mathsf{d}\dfrac{\mathsf{d}f(x)}{\mathsf{d}x}(y) \cdot z}{\mathsf{d}(y,z)}(a,c) \cdot (b,d) $
 \end{enumerate}

\begin{example}\label{ex:CDC} \normalfont Here are some well-known examples of Cartesian differential categories: 
\begin{enumerate}[{\em (i)}]
\item \label{ex:poly} Let $k\text{-}\mathsf{POLY}$ be the category whose objects are $n \in \mathbb{N}$ and where a map ${P: n \to m}$ is an $m$-tuple of polynomials in $n$ variables, that is, $P = \langle p_1(\vec x), \hdots, p_m(\vec x) \rangle$ with $p_i(\vec x) \in k[x_1, \hdots, x_n]$. Then $k\text{-}\mathsf{POLY}$ is a Cartesian $k$-differential category where the differential combinator is given by the standard differentiation of polynomials. So for $P = \langle p_1(\vec x), \hdots, p_m(\vec x) \rangle: n \to m$, its derivative $\mathsf{D}[P]: n \times n \to m$ is defined as the tuple of polynomials $\mathsf{D}[P] = \left(\mathsf{D}[p_1](\vec x, \vec y), \hdots,\mathsf{D}[p_m](\vec x, \vec y) \right)$ where $\mathsf{D}[p_j](\vec x, \vec y)$ is the sum of the partial derivatives of $p_j(\vec x)$, that is, $\mathsf{D}[p_j](\vec x, \vec y) = \sum^n_{i=1} \frac{\partial p_j}{\partial x_i}(\vec x) y_i \in k[x_1, \hdots, x_n, y_1, \hdots, y_n]$. 
\item \label{ex:smooth} Let $\mathbb{R}$ be the set of real numbers. Define $\mathsf{SMOOTH}$ as the category whose objects are the Euclidean spaces $\mathbb{R}^n$ and whose maps are smooth functions between them. $\mathsf{SMOOTH}$ is a Cartesian $\mathbb{R}$-differential category where the differential combinator is defined as the total derivative of a smooth function. For a smooth function ${F: \mathbb{R}^n \to \mathbb{R}^m}$, which is in fact an $m$-tuple $F = \langle f_1, \hdots, f_m \rangle$ of smooth functions $f_i: \mathbb{R}^n \to \mathbb{R}$, its derivative ${\mathsf{D}[F]: \mathbb{R}^n \times \mathbb{R}^n \to \mathbb{R}^m}$ is defined as $\mathsf{D}[F] = \left(\mathsf{D}[f_1](\vec x, \vec y), \hdots,\mathsf{D}[f_m](\vec x, \vec y) \right)$ where $\mathsf{D}[f_j]: \mathbb{R}^n \times \mathbb{R}^n \to \mathbb{R}$ is defined as the sum of partial derivatives of $f_j$, so $\mathsf{D}[f_j](\vec x, \vec y) = \sum^n_{i=1} \frac{\partial f_1}{\partial x_i}(\vec x) y_i$. Note that $\mathbb{R}\text{-}\mathsf{POLY}$ is a sub-Cartesian differential category of $\mathsf{SMOOTH}$. 
\item Important examples of Cartesian differential categories are the coKleisli categories of \textbf{differential categories} \cite{blute2006differential,Blute2019}. Briefly, a differential category is, amongst other things, a symmetric monoidal category $\mathbb{X}$ with a comonad $\oc$ which comes equipped with a natural transformation $\partial_A: \oc A \otimes A \to \oc A$, called the \textbf{deriving transformation} \cite[Def 7]{Blute2019}, that satisfies analogues of the basic rules of differentiation such as the chain rule and the product rule. CoKleisli maps $f: \oc A \to B$ are interpreted as smooth maps from $A$ to $B$, where the derivative of $f$ is given by pre-composing with the deriving transformation, $f \circ \partial_A: \oc A \otimes A \to B$ \cite[Def 2.3]{Blute2019}. This can be made precise by saying that for a differential category $\mathbb{X}$ with finite products, its coKleisli category $\mathbb{X}_\oc$ is a Cartesian differential category where the differential combinator $\mathsf{D}$ is defined using the deriving transformation $\partial$ \cite[Prop 3.2.1]{blute2009cartesian}. For more details on differential categories, we invite the reader to see \cite{blute2006differential,Blute2019,ehrhard2017introduction,garner2020cartesian,kerjean2023taylor}.
\end{enumerate} 
See \cite{cockett2020linearizing,garner2020cartesian} for lists of more examples of Cartesian differential categories. 
\end{example}

There are two important classes of maps in a Cartesian differential category that are worth mentioning: \emph{differential} linear maps \cite[Definition 2.2.1]{blute2009cartesian} and \emph{differential} constant maps \cite[Sec 6]{lemay2022properties}. Differential constants are maps whose derivative is zero, while differential linear maps are maps whose derivatives are themselves. 

\begin{definition} \label{def:dlin+con} In a Cartesian $k$-differential category, a map $f: A \to B$ is:
    \begin{enumerate}[{\em (i)}]
    \item a $\mathsf{D}$-constant \cite[Def 6.1]{lemay2022properties} if $\dfrac{\mathsf{d}f(x)}{\mathsf{d}x}(a) \cdot b =0$
\item $\mathsf{D}$-linear \cite[Def 2.7]{garner2020cartesian} if $\dfrac{\mathsf{d}f(x)}{\mathsf{d}x}(a) \cdot b = f(b)$. 
\end{enumerate}
\end{definition}

Properties of differential constants can be found in \cite[Sec 6]{lemay2022properties}, while properties of differential linear maps can be found in \cite[Lemma 2.6]{cockett2020linearizing}. 

Two essential operations which can be derived from the total derivative are \emph{partial} derivatives and \emph{higher-order} derivatives. Starting with partial derivative, given a map of, say, type $f: C_1 \times A \times C_2 \to B$, we'd like to take the derivative of $f$ with respect to $A$ while keeping the others constant. In classical differential calculus, partial derivatives are obtained by inserting zeroes in the appropriate vector argument of the total derivative. The same idea is true in Cartesian differential categories, where we obtain partial derivatives by inserting zeroes into the total derivative \cite[Sec 4.5]{blute2009cartesian}.

\begin{definition} In a Cartesian $k$-differential category, for a map $f: C_1 \times A \times C_2 \to B$, its \textbf{partial derivative} \cite[Def 2.7]{garner2020cartesian} in $A$ is the map $\mathsf{D}^{C_1 \times \_ \times C_2}[f]: C_1 \times A \times C_2 \times A \to B$, which in term calculus is written as:
 \[\mathsf{D}^{C_1 \times \_ \times C_2}[f](c_1,a_1,c_2,a_2) = \dfrac{\mathsf{d}f(c_1,x,c_2)}{\mathsf{d}x}(a_1) \cdot a_2\]
 and is defined as follows:
\[\dfrac{\mathsf{d}f(c_1,x,c_2)}{\mathsf{d}x}(a_1) \cdot a_2 = \dfrac{\mathsf{d}f(v,x,u)}{\mathsf{d}(v,x,u)}(c_1,a_1, c_2) \cdot (0,a_2,0)\]
Moreover, we say that $f: C_1 \times A \times C_2 \to B$ is: 
     \begin{enumerate}[{\em (i)}]
    \item $\mathsf{D}$-constant in $A$ if $\dfrac{\mathsf{d}f(c_1,x,c_2)}{\mathsf{d}x}(a_1) \cdot a_2 =0$
\item $\mathsf{D}$-linear in $A$ if $\dfrac{\mathsf{d}f(c_1,x,c_2)}{\mathsf{d}x}(a_1) \cdot a_2 = f(c_1,a_2,c_2)$. 
\end{enumerate}
\end{definition}

Let's now discuss \emph{higher-order derivatives}, which will play a central role in the story of this paper. For $n \in \mathbb{N}$, denote $A^{\times^n}$ as a shorthand for the product of $n$-copies of $A$, with the convention that $A^{\times^0} = \ast$ and $A^{\times^1} = A$.
Now for a map $f: A \to B$, applying the differential combinator $n$-times results in a map of type $\mathsf{D}^n[f]: A^{\times^{2^n}} \to B$ called the \textbf{$n$-th total derivative} of $f$. However, as explained in \cite[Sec 3.1]{garner2020cartesian}, due to the fact that total derivatives are equal to the sum of the partial derivatives \cite[Lemma 4.5.1]{blute2009cartesian} and also \textbf{[CD.6]}, there is a lot of extra redundant information in $\mathsf{D}^n[f]$. For example, the second total derivative can be worked out to be:
\[\mathsf{D}^2[f](a,b,c,d) = \dfrac{\mathsf{d}\dfrac{\mathsf{d}f(x)}{\mathsf{d}x}(y) \cdot b}{\mathsf{d}y}(a) \cdot c +  \dfrac{\mathsf{d}f(x)}{\mathsf{d}x}(a) \cdot d \]
So we see that $\mathsf{D}^2[f]$ has a $\mathsf{D}[f]$ summand -- which does not tell us any new information about $f$. Instead, all the new information comes from differentiating the first argument repeatedly. 

\begin{definition} In a Cartesian $k$-differential category, for a map $f: A \to B$ and every $n \in \mathbb{N}$, the \textbf{$n$-th derivative} \cite[Def 3.1]{garner2020cartesian} of $f$ is the map $\partial^{(n)}[f]: A \times A^{\times^n} \to B$, which is written in the term calculus as:
\[\partial^{(n)}[f](a_0, a_1, \hdots, a_n) := \dfrac{\mathsf{d}^{(n)} f(x)}{\mathsf{d}x}(a_0) \cdot a_1 \cdot \hdots \cdot a_n\]
and is defined inductively as: 
\begin{gather*}
  \dfrac{\mathsf{d}^{(0)} f(x)}{\mathsf{d}x}(a_0) = f(a_0) \qquad  \dfrac{\mathsf{d}^{(n+1)} f(x)}{\mathsf{d}x}(a_0) \cdot a_1 \cdot \hdots \cdot a_n \cdot a_{n+1}
 = \dfrac{\mathsf{d} \dfrac{\mathsf{d}^{(n)} f(x)}{\mathsf{d}x}(y) \cdot a_1 \cdot \hdots \cdot a_n  }{\mathsf{d}y} (a_0) \cdot a_{n+1}
\end{gather*}
\end{definition}

Here are now some of the main identities for higher-order derivatives, where in particular, we highlight that \textbf{[HD.1]} to \textbf{[HD.7]} are the higher-order versions of \textbf{[CD.1]} to \textbf{[CD.7]}. As such, \textbf{[HD.5]} is Faà di Bruno's Formula, which expresses the higher-order chain rule. Let's introduce some notation to help write down Faà di Bruno's Formula. For every $n \in \mathbb{N}$, let $[0]=\emptyset$ and let $[n+1] =\lbrace 1< \hdots < n+1 \rbrace$. Now for every subset $I = \lbrace i_1 < \hdots<  i_m \rbrace \subseteq [n+1]$, for a vector $x = (x_1, \hdots, x_{n+1})$, define $\vec x_I = (x_{i_1}, \hdots, x_{i_m})$. Lastly, we denote a \emph{non-empty} partition of $[n+1]$ as $[n+1] = A_1 \vert \hdots \vert A_k$, and let $\vert A_j \vert$ be the cardinality of $A_j$. Then Faà di Bruno's Formula \cite[Lemma 3.14]{garner2020cartesian} for the $n+1$-th derivative is given as a sum over the non-empty partitions of $[n+1]$. 

\begin{lemma}\label{lem:HD} \cite[Sec 3]{garner2020cartesian} In a Cartesian $k$-differential category: 
\begin{enumerate}[{\bf [HD.1]}]
\setcounter{enumi}{-1}
\item If $f$ is a $\mathsf{D}$-constant then $\partial^{n+1}[f]=0$ for all $n \in \mathbb{N}$;
\item  $\partial^n[r\cdot f + s \cdot g] = r \cdot \partial^n[f] + s \cdot \partial^n[g]$ for all $r,s \in k$ 
\item $\partial^n[f]$ is $k$-linear in each of its last $n$-arguments (i.e. is $k$-multilinear in its last $n$-arguments); 
\item If $f$ is $\mathsf{D}$-linear then $\partial^{n+2}[f]=0$ for all $n \in \mathbb{N}$; 
\item $\partial^n[\left\langle f_0, \hdots, f_n \right \rangle] = \left \langle  \partial^n[f_0], \hdots, \partial^n[f_n] \right \rangle$; 
\item The following equality holds: 
\begin{gather*}\dfrac{\mathsf{d}^{(n)} g(f(x))}{\mathsf{d}x}(a_0) \cdot a_1 \cdot \hdots \cdot a_{n} \\
= \sum \limits_{ [n]=A_1 \vert \hdots \vert A_k}  \dfrac{\mathsf{d}^{(k)} g(z)}{\mathsf{d}z}(a_0) \cdot \left( \dfrac{\mathsf{d}^{\left(  \vert A_1 \vert  \right)} f(x) }{\mathsf{d}x} (a_0) \cdot \vec a_{A_1} \right) \cdot \hdots \cdot  \left( \dfrac{\mathsf{d}^{\left(  \vert A_k \vert  \right)} f(x) }{\mathsf{d}x} (a_0) \cdot \vec a_{A_k} \right)
\end{gather*}
\item $\partial^n[f]$ is $\mathsf{D}$-linear in each of its last $n$-arguments (i.e. is $\mathsf{D}$-multilinear in its last $n$-arguments);
\item $\partial^n[f]$ is symmetric in its last $n$-arguments;
\item The following equalities hold: 
\begin{gather*}
    \dfrac{\mathsf{d} \dfrac{\mathsf{d}^{(n)}f(u)}{\mathsf{d}u}(x_0) \cdot x_1 \cdot \hdots\cdot x_n}{\mathsf{d}(x_0, x_1, \hdots, x_n)}(a_0, a_1, \hdots, a_n) \cdot (b_0, b_1, \hdots, b_n) \\
    =  \dfrac{\mathsf{d}^{(n)} \dfrac{\mathsf{d}f(u)}{\mathsf{d}u}(x) \cdot y}{\mathsf{d}(x,y)}(a_0,b_0) \cdot (a_1,b_1) \cdot \hdots \cdot (a_n,b_n) \\
    = \dfrac{\mathsf{d}^{(n+1)} f(x)}{\mathsf{d}x}(a_0) \cdot a_1 \cdot \hdots \cdot a_n \cdot b_0 + \sum\limits^n_{i=1} \dfrac{\mathsf{d}^{(n)} f(x)}{\mathsf{d}x}(a_0) \cdot a_1 \cdot \hdots a_{i-1} \cdot b_i \cdot a_{i+1} \cdot \hdots \cdot a_n
\end{gather*}
\end{enumerate}
\end{lemma}

\section{Differential Polynomials}\label{sec:DPOLY}

In this section, we introduce the concept of \emph{differential polynomials} in a Cartesian differential category, which will be a useful intermediate concept for the story of Taylor series. Classically, polynomials are defined, of course, using multiplication, sums, and scalar multiplication. However, while a Cartesian differential category has access to sums and scalar multiplication, an arbitrary Cartesian differential category may not have a proper notion of multiplication of maps. Instead, another way of characterizing polynomials amongst the smooth functions is that they are smooth functions whose $n$-th derivative is zero. This definition of polynomials can be easily given in a Cartesian differential category. 

\begin{definition} In a Cartesian $k$-differential category, a \textbf{$\mathsf{D}$-polynomial} is a map $p: A \to B$ such that there is a $n \in \mathbb{N}$ such that $\partial^{(n+1)}[p] =0$. For a $\mathsf{D}$-polynomial $p$, the smallest $n \in \mathbb{N}$ such that $\partial^{(n+1)}[p] =0$ is called the \textbf{$\mathsf{D}$-degree} of $p$, which we denote as $\mathsf{deg}(p)$.
\end{definition}

Differential polynomials are closed under all the basic operations of a Cartesian differential category. As such, the differential polynomials of a Cartesian differential category form a Cartesian differential category. 

\begin{lemma}\label{lem:Dpoly} In a Cartesian $k$-differential category,  
    \begin{enumerate}[{\em (i)}]
    \item \label{Dpoly-0} $\mathsf{D}$-constants are $\mathsf{D}$-polynomials of $\mathsf{D}$-degree $0$; 
\item \label{Dpoly-1} $\mathsf{D}$-linear maps are $\mathsf{D}$-polynomials, and non-zero $\mathsf{D}$-linear maps have $\mathsf{D}$-degree $1$; 
\item \label{Dpoly-2} Identity maps and projections maps are $\mathsf{D}$-polynomials, and for non-terminal objects they have $\mathsf{D}$-degree $1$;
\item \label{Dpoly-3} If $p: A \to B$ and $q: A \to B$ are $\mathsf{D}$-polynomials, then $r \cdot p + s \cdot q$ is a $\mathsf{D}$-polynomial whose $\mathsf{deg}(r \cdot p+ s\cdot q) \leq \mathsf{max}\lbrace \mathsf{deg}(p), \mathsf{deg}(q) \rbrace$;
\item \label{Dpoly-5} If $p_j: A \to B_j$ are $\mathsf{D}$-polynomials, then $\langle p_0, \hdots, p_n \rangle: A \to B_0 \times \hdots \times B_n$ is a $\mathsf{D}$-polynomial where $\mathsf{deg}(\langle p_0, \hdots, p_n \rangle) = \mathsf{max}\lbrace \mathsf{deg}(p_j) \rbrace$;
\item \label{Dpoly-4} If $p: A \to B$ and $q: B \to C$ are $\mathsf{D}$-polynomials, then $q \circ p$ is a $\mathsf{D}$-polynomial where $\mathsf{deg}(p \circ q) \leq \mathsf{deg}(p)\mathsf{deg}(q)$;
\item \label{Dpoly-6} If $p: A \to B$ is a $\mathsf{D}$-polynomial, then $\mathsf{D}[p]: A \times A \to B$ is a $\mathsf{D}$-polynomial where $\mathsf{deg}\left( \mathsf{D}[p] \right) = \mathsf{deg}(p)$. 
\end{enumerate}
\end{lemma}
\begin{proof} These follow mostly from Lemma \ref{lem:HD}. So (\ref{Dpoly-0}) and (\ref{Dpoly-1}) are immediate from \textbf{[HD.0]} and \textbf{[HD.3]} respectively, and since \textbf{[CD.3]} tells us that identity maps and projections are $\mathsf{D}$-linear, (\ref{Dpoly-2}) follows from (\ref{Dpoly-1}). Both (\ref{Dpoly-3}) and (\ref{Dpoly-5}) are straightforward to check using \textbf{[HD.1]} and \textbf{[HD.4]} respectively. The remaining two require a bit more explanation. So for (\ref{Dpoly-4}), let $\mathsf{deg}(p)=n$ and $\mathsf{deg}(q)=m$. Consider $\partial^{(mn+1)}[q \circ p]$. Faà di Bruno's Formula \textbf{[HD.5]} tells us that $\partial^{(mn+1)}[q \circ p]$ is a sum indexed by non-empty partitions of $[mn+1]=A_1 \vert \hdots \vert A_k$, so we have that $1 \leq k \leq mn+1$. If $m< k$, then $\partial^{(k)}[q] =0$, so all the summands of partitions of size $m < k$ are zero. If $k \leq m$, then $n \leq \frac{mn+1}{k}$, which means that in the partition, there is some $A_j$ with $n < \vert A_j \vert$. So $\partial^{\vert A_j \vert}[p]=0$, and since by \textbf{[HD.2]} $\partial^{(k)}[q]$ is $k$-multilinear in its last $k$-arguments means that we will get zero for the summand of the partition. So we conclude that $\partial^{(mn+1)}[q \circ p]=0$, and thus $q \circ p$ is a $\mathsf{D}$-polynomial. Lastly for (\ref{Dpoly-6}), suppose that $\mathsf{deg}(p)=n$. Now \textbf{[HD.8]} essentially tells us that $\mathsf{D}$ and $\partial^{(-)}$ commute, up to permutation. So since $\partial^{(n+1)}[p]=0$, we get that $\partial^{(n+1)}[\mathsf{D}[p]]=0$ and so $\mathsf{D}[p]$ is a $\mathsf{D}$-polynomial with at least $\mathsf{deg}\left( \mathsf{D}[p] \right) \leq n$. Now suppose that $\mathsf{deg}\left( \mathsf{D}[p] \right)=k < n$, so $\partial^{(k+1)}[\mathsf{D}[p]]=0$. However, note that by \textbf{[HD.8]} and \textbf{[HD.2]}, we can recover $\partial^{(k+1)}[p]$ from $\partial^{(k+1)}[\mathsf{D}[p]]$ in the following way: 
\begin{gather*}
\dfrac{\mathsf{d}^{(k+1)} \dfrac{\mathsf{d}f(u)}{\mathsf{d}u}(x) \cdot y}{\mathsf{d}(x,y)}(a_0,0) \cdot (0,a_1) \cdot (a_2,0) \cdot \hdots \cdot (a_{k+1},0) = \dfrac{\mathsf{d}^{(k+1)} f(x)}{\mathsf{d}x}(a_0) \cdot a_1 \cdot \hdots \cdot a_{k+1}
\end{gather*}
However since by assumption $\partial^{(k+1)}[\mathsf{D}[p]]=0$, it follows that $\partial^{(k+1)}[p]=0$. But this is a contradiction by the definition of $\mathsf{deg}(p)=n$. So we must have that $\mathsf{deg}\left( \mathsf{D}[p] \right) = \mathsf{deg}(p)$, as desired. 
\end{proof}

\begin{corollary} For a Cartesian $k$-differential category $\mathbb{X}$, let $\mathsf{D}\text{-}\mathsf{POLY}[\mathbb{X}]$ be the sub-category of $\mathsf{D}$-polynomials of $\mathbb{X}$. Then $\mathsf{D}\text{-}\mathsf{POLY}[\mathbb{X}]$ is a sub-Cartesian $k$-differential category of $\mathbb{X}$. 
\end{corollary}

\begin{example}\label{ex:DPOLY} \normalfont Here are the differential polynomials in our main examples: 
\begin{enumerate}[{\em (i)}]
\item Every map in $k\text{-}\mathsf{POLY}$ is a $\mathsf{D}$-polynomial, so $\mathsf{D}\text{-}\mathsf{POLY}[k\text{-}\mathsf{POLY}]=k\text{-}\mathsf{POLY}$. However, the $\mathsf{D}$-degree may not necessarily align with a polynomial's usual degree. Indeed, consider the case for $k=\mathbb{Z}_2$, the integers modulo $2$. Then for $p(x)= x^2$, we get that $\mathsf{D}[p](x,y) = 2xy =0$. So $x^2$ is a $\mathsf{D}$-constant in $\mathbb{Z}_2\text{-}\mathsf{POLY}$ so $\mathsf{deg}(x^2) =0$, which is clearly different from its usual degree of $2$. However, if $k$ is an integral domain (which recall means for $r,s \in k$, if $rs=0$ then $r=0$ or $s=0$) and has characteristic zero (which recall means that for all $n \in \mathbb{N}$, for $n+1 \in k$ we have that $n+1 \neq 0$), then the $\mathsf{D}$-degree does match the usual polynomial degree. 
\item The $\mathsf{D}$-polynomials in $\mathsf{SMOOTH}$ are precisely tuples of usual real polynomial functions, so $\mathsf{D}\text{-}\mathsf{POLY}[\mathsf{SMOOTH}]=\mathbb{R}\text{-}\mathsf{POLY}$.
\item The $\mathsf{D}$-polynomials in the coKleisli category of a differential category with finite products correspond precisely to the notion of polynomials in a differential category as defined by Ehrhard in \cite[Sec 4.1.1]{ehrhard2017introduction}. So define $\partial^{(n)}_A: \oc A \otimes A^{\otimes^n} \to \oc A$ inductively as $\partial^{(0)}_A=1_{\oc A}$ and $\partial^{(n+1)}_A= \partial_A \circ (\partial^{(n)}_A \otimes 1_A)$. Then $f: \oc A \to B$ is a $\mathsf{D}$-polynomial if and only if $f \circ \partial^{(n+1)}_A=0$. However, in \cite[Sec 4.1.1]{ehrhard2017introduction}, Ehrhard uses another formula for the composition of $\mathsf{D}$-polynomials, which involves positive rationals, and recaptures the notion of linear application in Differential Linear Logic \cite{pagani2009inverse,boudes2013characterization,ehrhard2008uniformity}. Here, by Lemma \ref{lem:Dpoly}, we get that the usual coKleisli composition of $\mathsf{D}$-polynomials is again a $\mathsf{D}$-polynomial, and so $\mathsf{D}\text{-}\mathsf{POLY}[\mathbb{X}_\oc]$ is an actual subcategory of the coKleisli category. While in well-behaved models, the coKleisli composition and Ehrhard's composition may correspond, in general, they need not be equal. 
\end{enumerate} 
\end{example}

\section{Taylor Differentials Polynomials}\label{sec:TPOLY}

In this section, we introduce the notion of \emph{Taylor differential monomials} and \emph{Taylor differential polynomials} of maps in a Cartesian differential category. 

In differential calculus, for a smooth function $f: \mathbb{R} \to \mathbb{R}$, its $n$-th Taylor monomial is defined as the $n$-th term in its Taylor series, that is, $\frac{1}{n!}\cdot f^{(n)}(0)x^n$, where $f^{(n)}$ is the $n$-th derivative of $f$. Then its $n$-th Taylor polynomial is defined as the sum of all $k\leq n$-th Taylor monomials, that is, $\sum^n_{k=0} \frac{1}{k!}\cdot f^{(k)}(0)x^k$. So, how do we generalize these formulas in a Cartesian differential category? The first thing to address is scalar multiplication by $\frac{1}{n!}$. So we need to be in a setting where we can scalar multiply by positive rationals. As such, we will need to assume that in our base commutative semiring $k$, every $n= 1 + \hdots +1 \in k$ is invertible. This amounts to saying that $k$ is a $\mathbb{Q}_{\geq 0}$-algebra, where $\mathbb{Q}_{\geq 0}$ is the semiring of positive rationals. We next need to address the $x^n$ part. As mentioned above, arbitrary Cartesian differential categories do not necessarily have a multiplication operation for maps. Luckily, this issue is solved thanks to the differential combinator. Indeed, note that $\partial^{(n)}[f]: \mathbb{R} \times \mathbb{R}^{\times^n} \to \mathbb{R}$ is $\partial^{(n)}[f](x_0, x_1, \hdots, x_n) =  f^{(n)}(x_0)x_1\hdots x_n$. So by inserting zero into the first argument and copying the second argument $n$ times, we get $\partial^{(n)}[f](0,x,\dots, x) = f^{(n)}(0)x^n$. So we recapture the $n$-th Taylor monomial using the differential combinator via $\frac{1}{n!} \cdot \partial^{(n)}[f](0,x,\dots, x)=\frac{1}{n!}\cdot f^{(n)}(0)x^n$. This is the formula we generalize to define Taylor differential monomials in the Cartesian differential setting. The sum of these Taylor differential monomials gives us the Taylor differential polynomials. 

For the remainder of this section, we assume that $k$ is a commutative $\mathbb{Q}_{\geq 0}$-algebra.

\begin{definition} In a Cartesian $k$-differential category, for a map $f: A \to B$ and every $n \in \mathbb{N}$, define $\mathcal{M}^{(n)}[f]: A \to B$ as the composite:
\[\mathcal{M}^{(n)}[f] := \dfrac{1}{n!} \cdot \partial^{(n)}[f] \circ \langle 0, 1_A, \hdots, 1_A \rangle\]
which in the term calculus is written as:
\[\mathcal{M}^{(n)}[f](x) = \dfrac{1}{n!} \cdot \dfrac{\mathsf{d}^{(n)} f(u)}{\mathsf{d}u}(0) \cdot x \cdot \hdots \cdot x\] 
and then define $\mathcal{T}^{(n)}[f]: A \to B$ as:
\[\mathcal{T}^{(n)}[f] = \sum^n_{k=0} \mathcal{M}^{(k)}[f]\] 
\end{definition}

We note that $\mathcal{M}^{(0)}[f]$ is given by evaluating $f$ at $0$, so $\mathcal{M}^{(0)}[f](a) = f(0)$, while $\mathcal{M}^{(1)}[f]$ is precisely the \textbf{linearization} of $f$ (also sometimes written as $\mathsf{L}[f]$), which is $\mathcal{M}^{(1)}[f](a) = \dfrac{\mathsf{d} f(x)}{\mathsf{d}x}(0) \cdot a$ \cite[Prop 3.6]{cockett2020linearizing}. 

\begin{example}\label{ex:DMONO} \normalfont Here are the Taylor differential monomials/polynomials of maps in our main examples: 
\begin{enumerate}[{\em (i)}]
\item  The $k$-th Taylor $\mathsf{D}$-monomial of a tuple of polynomials $P=\langle p_i(\vec x) \rangle^m_{i=1}: n \to m$ is $\mathcal{M}^{(k)}[ P] = \langle \mathcal{M}^{(k)}[p_i](\vec x)\rangle^m_{i=1}$ where $\mathcal{M}^{(k)}[p_i](\vec x) \in k[x_1, \hdots, x_n]$ are the all the degree $k$ monomials summands of $p_i(\vec x)$. Then the $j$-th Taylor $\mathsf{D}$-monomial of $P$ is $\mathcal{T}^{(j)}[ P] = \langle \mathcal{T}^{(j)}[p_i](\vec x)\rangle^m_{i=1}$ where $\mathcal{T}^{(j)}[p_i](\vec x) \in k[x_1, \hdots, x_n]$ is the sum of all degree $k \leq j$ monomial summands of $p_j(\vec x)$. In particular, for $p(x)=\sum^n_{k=0} a_k x^k$, we get that $\mathcal{M}^{(k)}[p](\vec x) = a_k x^k$ and $\mathcal{T}^{(j)}[p](x) = \sum^j_{k=0} a_k x^k$. 
\item For a smooth function $F= \langle f_i \rangle^m_{i=1}: \mathbb{R}^n \to \mathbb{R}^m$, its $k$-th Taylor $\mathsf{D}$-monomial is $\mathcal{M}^{(k)}[F](\vec x) = \langle \mathcal{M}^{(k)}[f_i](\vec x) \rangle^m_{i=1}$ where $\mathcal{M}^{(k)}[f_i](\vec x)$ is the usual $k$-th Taylor monomial of $f_i$ in the multivariable case:  
\[ \mathcal{M}^{(k)}[f_m](\vec x) =  \sum\limits_{\lbrace i_1, \hdots, i_k \rbrace \subseteq \lbrace 1, \hdots, n \rbrace} \frac{1}{k!}\cdot \frac{\partial^{(k)} f_m}{\partial x_{i_1} \hdots \partial x_{i_k} }(\vec 0)x_{i_1}\hdots x_{i_k} \]
Then the $j$-th Taylor $\mathsf{D}$-polynomial of $F$ is the tuple of the usual $j$-th Taylor polynomials of  $f_i$ in the multivariable case. In particular, for a smooth function $f: \mathbb{R} \to \mathbb{R}$, we get precisely that $\mathcal{M}^{(k)}[f](x) = \frac{1}{k!}\cdot f^{(k)}(0)x^k$ and $\mathcal{T}^{(j)}[f](x) = \sum^j_{k=0} \frac{1}{k!}\cdot f^{(k)}(0)x^k$. 
\item In a differential category, using the contraction and dereliction natural transformations, we obtain natural transformations of type $\oc A \to \oc A \otimes A^{\otimes^n}$. Post-composing this with the natural transformation $\partial^{(n)}_A: \oc A \otimes A^{\otimes^n} \to \oc A$ gives us the natural transformation $\mathsf{M}^{(n)}_A: \oc A \to \oc A$ as defined in \cite[Sec III.E]{kerjean2023taylor}. Then define the natural transformation $\mathsf{T}^{(n)}_A: \oc A \to \oc A$ as the sum $\mathsf{T}^{(n)}_A = \sum^n_{k=0} \mathsf{M}^{(k)}_A$, as defined in \cite[Sec 3.1]{ehrhard2017introduction}. Then in the coKleisli category, the $n$-th Taylor $\mathsf{D}$-monomial of a coKleisli map $f: \oc A \to B$ is worked out to be precisely $\mathcal{M}^{(n)}[ f ] = f \circ \mathsf{M}^{(n)}_A$, while its $n$-th Taylor $\mathsf{D}$-polynomial is $\mathcal{T}^{(n)}[ f ] = f \circ \mathsf{T}^{(n)}_A$. The natural transformations $\mathsf{M}^{(n)}$ and the Taylor $\mathsf{D}$-monomials of coKleisli maps played a central role in the story of \cite{kerjean2023taylor}, while the natural transformations $\mathsf{T}^{(n)}$ and the Taylor $\mathsf{D}$-polynomials of coKleisli maps were studied in detail in \cite{ehrhard2017introduction}. 
\end{enumerate} 
\end{example}

Here are some useful identities about Taylor differential monomials and Taylor differential polynomials:  

\begin{lemma}\label{lemma:monomial} In a Cartesian $k$-differential category: 
    \begin{enumerate}[{\em (i)}]
    \item \label{M-1} $\mathcal{M}^{(0)}[f]$ is a $\mathsf{D}$-constant;
    \item \label{M-2} $\mathcal{M}^{(1)}[f]$ is $\mathsf{D}$-linear;
    \item\label{M-3} $\mathcal{M}^{(n)}[ r \cdot f + s \cdot g ] = r \cdot \mathcal{M}^{(n)}[  f ] +  s \cdot \mathcal{M}^{(n)}[ g ]$;
    \item \label{M-9} $\mathcal{M}^{(n)}[f](r\cdot a + s \cdot b) = \frac{1}{n!} \cdot \sum\limits^n_{k=0} \frac{r^k s^{n-k}}{k! (n-k)!}\cdot \dfrac{\mathsf{d}^{(n+1)} f(x)}{\mathsf{d}x}(0) \cdot \underbrace{a \cdot \hdots \cdot a}_{\text{$k$-times}} \cdot \underbrace{b \cdot \hdots \cdot b}_{\text{$n-k$-times}}$
    \item \label{M-4} $\mathcal{M}^{(n)}\left[ \langle f_0, \hdots, f_n \rangle \right] = \left \langle \mathcal{M}^{(n)}[f_0], \hdots, \mathcal{M}^{(n)}[f_n] \right \rangle$; 
    \item \label{M-5} $\mathcal{M}^{(n)}[ g \circ f ](a) = \sum \limits_{n=m_1 + \hdots + m_k} \frac{1}{k!m_1!\hdots m_k!}\dfrac{\mathsf{d}^{(k)} g(z)}{\mathsf{d}z}(f(0)) \cdot \mathcal{M}^{(m_1)}[f](a) \cdot \hdots \cdot \mathcal{M}^{(m_k)}[f](a)$;
    \item \label{M-10} $\mathcal{M}^{(n)}\left[ \mathsf{D}[f] \right](a,b) = \dfrac{\mathsf{d} \mathcal{M}^{(n)}[f](x)}{\mathsf{d}x} (a) \cdot b$
    \item \label{M-6} $\dfrac{\mathsf{d} \mathcal{M}^{(n+1)}[f](x)}{\mathsf{d}x} (a) \cdot b = \frac{1}{n!} \cdot \dfrac{\mathsf{d}^{(n+1)} f(x)}{\mathsf{d}x}(0) \cdot \underbrace{a \cdot \hdots \cdot a}_{\text{$n$-times}} \cdot b$; 
    \item \label{M-7} $\mathcal{M}^{(n)}[f]$ is a $\mathsf{D}$-polynomial where $\mathsf{deg}\left( \mathcal{M}^{(n)}[f] \right) = n$ if $\mathcal{M}^{(n)}[f] \neq 0$ and $\mathsf{deg}\left( \mathcal{M}^{(n)}[f] \right) = 0$ if $\mathcal{M}^{(n)}[f] = 0$;
    \item \label{M-8} $\mathcal{M}^{(n)}\left[\mathcal{M}^{(n)}[ f ] \right] = \mathcal{M}^{(n)}[f]$ and $\mathcal{M}^{(m)}\left[\mathcal{M}^{(n)}[ f ] \right] = 0$ when $m\neq n$;
    \item \label{T-1} $\mathcal{M}^{(k)}\left[\mathcal{T}^{(n)}[ f ] \right] = \mathcal{M}^{(k)}[f]$ for $k\leq n$ and $\mathcal{M}^{(k)}\left[\mathcal{T}^{(n)}[ f ] \right] = 0$ if $n < k$;
    \item \label{T-2} $\mathcal{T}^{(n)}\left[\mathcal{M}^{(k)}[ f ] \right] = \mathcal{M}^{(k)}[f]$ for $k\leq n$ and $\mathcal{T}^{(n)}\left[\mathcal{M}^{(k)}[ f ] \right] = 0$ if $n < k$
    \item \label{T-3} $\mathcal{T}^{(n)}\left[\mathcal{T}^{(n)}[ f ] \right] = \mathcal{T}^{(n)}[f]$, and also $\mathcal{T}^{(k)}\left[\mathcal{T}^{(n)}[ f ] \right] = \mathcal{T}^{(k)}[f]$ for $k\leq n$ and $\mathcal{T}^{(k)}\left[\mathcal{T}^{(n)}[ f ] \right] = \mathcal{T}^{(n)}[f]$ for $n < k$; 
    \item \label{T-4} $\mathcal{T}^{(n)}[f]$ is a $\mathsf{D}$-polynomial where $\mathsf{deg}\left( \mathcal{T}^{(n)}[f] \right) \leq n$ if $\mathcal{T}^{(n)}[f] \neq 0$ and $\mathsf{deg}\left( \mathcal{T}^{(n)}[f] \right) = 0$ if $\mathcal{T}^{(n)}[f] = 0$
\end{enumerate}
\end{lemma}
\begin{proof} Now (\ref{M-1}) follows from \cite[Lem 6.3]{lemay2022properties}, while (\ref{M-2}) is given by \cite[Lem 3.6.(i)]{lemay2022properties}. It is straightforward to compute that (\ref{M-3}) follows from \textbf{[HD.1]}, (\ref{M-9}) follows from \textbf{[HD.2]} and the binomial theorem, (\ref{M-4}) follows from \textbf{[HD.4]}, (\ref{M-5}) follows from \textbf{[HD.5]}, while (\ref{M-10}) and (\ref{M-6}) follows from \textbf{[HD.8]}. For (\ref{M-7}), by Lem \ref{lem:Dpoly}.(\ref{Dpoly-0}) and (\ref{M-1}), we have that $\mathcal{M}^{(0)}[f]$ is a $\mathsf{D}$-polynomial with $\mathsf{deg}\left( \mathcal{M}^{(0)}[f] \right) = 0$ always. For the $n+1$ case, from (\ref{M-6}) we can compute that for $0 \leq k \leq n+1$: 
\begin{align}\label{Mpartial1}
    \dfrac{\mathsf{d}^{(k)} \mathcal{M}^{(n+1)}[f](x)}{\mathsf{d}x} (a) \cdot b_1 \cdot \hdots \cdot b_{k} = \frac{1}{(n+1-k)!} \cdot \dfrac{\mathsf{d}^{(n+1)} f(x)}{\mathsf{d}x}(0) \cdot \underbrace{a \cdot \hdots \cdot a}_{\text{$n+1-k$-times}} \cdot b_1 \cdot \hdots \cdot b_{k}
\end{align}
So for $k=n+1$, we get that:
\begin{align}\label{Mpartial2}
    \dfrac{\mathsf{d}^{(n+1)} \mathcal{M}^{(n+1)}[f](x)}{\mathsf{d}x} (a) \cdot b_1 \cdot \hdots \cdot b_{n+1} = \dfrac{\mathsf{d}^{(n+1)} f(x)}{\mathsf{d}x}(0) \cdot b_1 \cdot \hdots \cdot b_{n+1}
\end{align}
Note that the right-hand side does not depend on the $a$, so it follows that $\partial^{(n+1)}\left[  \mathcal{M}^{(n+1)}[f] \right]$ is $\mathsf{D}$-constant in its first argument. So we get that $\partial^{(n+2)}\left[  \mathcal{M}^{(n+1)}[f] \right]=0$. Thus $\mathcal{M}^{(n+1)}[f]$ is a $\mathsf{D}$-polynomial where at least $\mathsf{deg}\left( \mathcal{M}^{(n+1)}[f] \right) \leq n+1$. If $\mathcal{M}^{(n+1)}[f] =0$, then since $0$ is a $\mathsf{D}$-constant, $\mathsf{deg}\left( \mathcal{M}^{(n+1)}[f] \right) = 0$. So suppose that $\mathcal{M}^{(n+1)}[f] \neq 0$. Now suppose that $\mathsf{deg}\left( \mathcal{M}^{(n+1)}[f] \right) = k < n+1$, then we have that $\partial^{(k+1)}\left[  \mathcal{M}^{(n+1)}[f] \right]=0$. However from (\ref{Mpartial1}), we get that: 
\begin{gather*}
    0 = \frac{1}{(n+1)n \hdots (n+2-k} \cdot 0 = \frac{1}{(n+1) n \hdots (n-k)} \cdot \dfrac{\mathsf{d}^{(k+1)} \mathcal{M}^{(n+1)}[f](x)}{\mathsf{d}x} (0) \cdot a \cdot \hdots \cdot a \\
    =  \frac{1}{(n+1) n \hdots (n+2-k)} \cdot \frac{1}{(n+1-k)!} \cdot \dfrac{\mathsf{d}^{(n+1)} f(x)}{\mathsf{d}x}(0) \cdot a \cdot \hdots \cdot a = \mathcal{M}^{(n+1)}[f](a)  
\end{gather*}
But $\mathcal{M}^{(n+1)}[f](a) \neq 0$, so we get a contradiction. Therefore, we must also have $n +1 \leq \mathsf{deg}\left( \mathcal{M}^{(n+1)}[f] \right)$. So we get that $\mathsf{deg}\left( \mathcal{M}^{(n+1)}[f] \right) = n+1$ as desired. Now for (\ref{M-8}), the case when $n=0$ clearly holds. For the $n+1$, from (\ref{Mpartial2}), we compute that:
\begin{gather*}
  \mathcal{M}^{(n+1)}\left[\mathcal{M}^{(n+1)}[ f ] \right](a)= \frac{1}{n!} \cdot \dfrac{\mathsf{d}^{(n+1)} \mathcal{M}^{(n+1)}[f](x)}{\mathsf{d}x} (0) \cdot a \cdot \hdots \cdot a \\
  = \frac{1}{n!} \cdot \dfrac{\mathsf{d}^{(n+1)} f(x)}{\mathsf{d}x}(0) \cdot a \cdot \hdots \cdot a = \mathcal{M}^{(n+1)}[ f ](a). 
\end{gather*}
So $\mathcal{M}^{(n+1)}\left[\mathcal{M}^{(n+1)}[ f ] \right] = \mathcal{M}^{(n+1)}[ f ]$ as desired. On the other hand, suppose that $m \neq n+1$. When $n+1 < m$, from (\ref{M-7}) we get that $\mathcal{M}^{(m)}\left[\mathcal{M}^{(n+1)}[ f ] \right] = 0$. When $m \leq n+1$, using (\ref{Mpartial1}) and \textbf{[HD.2]}, we compute that:
\begin{gather*}
 \mathcal{M}^{(m)}\left[\mathcal{M}^{(n+1)}[ f ] \right](a) = \frac{1}{m!} \cdot   \dfrac{\mathsf{d}^{(m)} \mathcal{M}^{(n+1)}[f](x)}{\mathsf{d}x} (0) \cdot a \cdot \hdots \cdot a\\
 = \frac{1}{m!} \cdot \frac{1}{(n+1-m)!} \cdot \dfrac{\mathsf{d}^{(n+1)} f(x)}{\mathsf{d}x}(0) \cdot \underbrace{0 \cdot \hdots \cdot 0}_{\text{$n+1-k$-times}} \cdot b_1 \cdot \hdots \cdot b_{k} = 0
\end{gather*}
So $\mathcal{M}^{(m)}\left[\mathcal{M}^{(n+1)}[ f ] \right] = 0$ as desired. Then (\ref{T-1}), (\ref{T-2}), and (\ref{T-3}) follow from (\ref{M-3}) and (\ref{M-8}), while (\ref{T-4}) follows from (\ref{M-7}) and Lemma \ref{lem:Dpoly}.(\ref{Dpoly-3}). 
\end{proof}

We can also define what we mean by maps being Taylor differential polynomials. 

\begin{definition} In a Cartesian $k$-differential category, 
\begin{enumerate}[{\em (i)}]
\item A \textbf{Taylor $\mathsf{D}$-monomial} is a map $m: A \to B$ such that for some $n \in \mathbb{N}$, $\mathcal{M}^{(n)}[m] = m$;
    \item A \textbf{Taylor $\mathsf{D}$-polynomial} is a map $p: A \to B$ such that for some $n \in \mathbb{N}$, $\mathcal{T}^{(n)}[p] = p$. 
    \end{enumerate}
\end{definition}

Note that the $n$ in the definition of Taylor differential monomials/polynomials is necessarily unique for $f\neq 0$. The terminology is justified by the fact that by Lemma \ref{lem:Dpoly}.(\ref{T-4}), every Taylor differential polynomial is also a differential polynomial. Taylor differential polynomials are also closed under the Cartesian $k$-differential structure. 

\begin{lemma}\label{lemma:Tpoly1} In a Cartesian $k$-differential category: 
    \begin{enumerate}[{\em (i)}]
    \item For every $f$, $\mathcal{M}^{(n)}[f]$ is a Taylor $\mathsf{D}$-monomial and $\mathcal{T}^{(n)}[f]$ is a Taylor $\mathsf{D}$-polynomial;
    \item Taylor $\mathsf{D}$-monomials are Taylor $\mathsf{D}$-polynomials;
\item If $f$ is $\mathsf{D}$-linear, then $f$ is a Taylor $\mathsf{D}$-monomial where $\mathcal{T}^{(1)}[f]=\mathcal{M}^{(1)}[f]=f$; 
\item Identity maps and projections maps are Taylor $\mathsf{D}$-monomials; 
\item If $p: A \to B$ and $q: A \to B$ are Taylor $\mathsf{D}$-polynomials, then $r \cdot p + s \cdot q$ is a Taylor $\mathsf{D}$-polynomial;
\item If $p_j: A \to B_j$ are Taylor $\mathsf{D}$-polynomials, then $\langle p_0, \hdots, p_n \rangle$ is a Taylor $\mathsf{D}$-polynomial; 
\item If $p: A \to B$ and $q: B \to C$ are Taylor $\mathsf{D}$-polynomials, then $q \circ p$ is a Taylor $\mathsf{D}$-polynomial;
\item If $p: A \to B$ is a Taylor $\mathsf{D}$-polynomial, then $\mathsf{D}[p]: A \times A \to B$ is a Taylor $\mathsf{D}$-polynomial. 
\end{enumerate}
\end{lemma}
\begin{proof} These follow immediately from the identities in Lemma \ref{lem:Dpoly}. 
\end{proof}

\begin{corollary} For a Cartesian $k$-differential category $\mathbb{X}$, let $\mathcal{T}\mathsf{D}\text{-}\mathsf{POLY}[\mathbb{X}]$ be the sub-category of Taylor $\mathsf{D}$-polynomials of $\mathbb{X}$. Then $\mathcal{T}\mathsf{D}\text{-}\mathsf{POLY}[\mathbb{X}]$ is a sub-Cartesian $k$-differential category of $\mathbb{X}$ (and of $\mathsf{D}\text{-}\mathsf{POLY}[\mathbb{X}]$). 
\end{corollary}

\begin{example}\label{ex:TPOLY} \normalfont Here are the Taylor differential polynomials in our main examples: 
\begin{enumerate}[{\em (i)}]
\item Every map in $k\text{-}\mathsf{POLY}$ is a Taylor $\mathsf{D}$-polynomial, so we have that $\mathsf{TD}\text{-}\mathsf{POLY}[k\text{-}\mathsf{POLY}]=k\text{-}\mathsf{POLY}=\mathsf{D}\text{-}\mathsf{POLY}[k\text{-}\mathsf{POLY}]$. The Taylor $\mathsf{D}$-monomials are the tuples of polynomials whose summands all have the same degree. 
\item The Taylor $\mathsf{D}$-polynomials in $\mathsf{SMOOTH}$ are precisely tuples of real polynomial functions, so $\mathsf{TD}\text{-}\mathsf{POLY}[\mathsf{SMOOTH}]=\mathbb{R}\text{-}\mathsf{POLY}= \mathsf{D}\text{-}\mathsf{POLY}[\mathsf{SMOOTH}]$.
\item A coKleisli map $f: \oc A \to B$ is a Taylor $\mathsf{D}$-monomial if there exists an $n$ such that $f \circ \mathsf{M}^{(n)}_A =f$, and a Taylor $\mathsf{D}$-polynomial if there exists an $n$ such that $f \circ \mathsf{T}^{(n)}_A =f$.
\item In the polynomial and smooth functions examples, the Taylor $\mathsf{D}$-polynomials and $\mathsf{D}$-polynomials coincide. However, this is not necessarily true in an arbitrary Cartesian differential category. A source of counter-examples comes from looking at the $\mathsf{D}$-constants. Now the only Taylor $\mathsf{D}$-polynomials that are also $\mathsf{D}$-constants are precisely the maps $c$ which are constant in the sense that $c \circ 0 =c$ (or in other words, $c = \mathcal{M}^{(0)}[c] = \mathcal{T}^{(0)}[c]$). However, as explained in \cite{lemay2022properties}, there are examples of Cartesian differential categories that have $\mathsf{D}$-constants that are not constant. Indeed, every \emph{cofree} Cartesian differential category \cite{cockett2011faa,garner2020cartesian,lemay2022properties} have $\mathsf{D}$-constants that are not constants. Therefore in a cofree Cartesian differential category, there are $\mathsf{D}$-polynomials that are not Taylor $\mathsf{D}$-polynomials. 
\end{enumerate} 
\end{example}

\section{Ultrametric for Taylor Series Convergence}\label{sec:ultra}

In this section, we introduce a canonical ultrapseudometric for Cartesian differential categories, which is defined by comparing Taylor differential monomials of maps. We then show our main result that if this ultrapseudometric is an ultrametric, the Taylor series of a map converges to said map. For the remainder of this section, we again assume that $k$ is a commutative $\mathbb{Q}_{\geq 0}$-algebra.

So let $\mathbb{R}_{\geq 0}$ be the set of positive real numbers. Then recall that an \textbf{ultrapseudometric} on a set $M$ is a function $\mathsf{d}: M \times M \to \mathbb{R}_{\geq 0}$ such that: (1) $\mathsf{d}(x,x)=0$, (2) $\mathsf{d}(x,y) = \mathsf{d}(y,x)$, and (3) $\mathsf{d}(x,z) \leq \mathsf{max}\lbrace \mathsf{d}(x,y), \mathsf{d}(y,z)  \rbrace$. This latter inequality is called the \textbf{Strong Triangle Inequality}. An \textbf{ultrapseudometric space} is a pair $(M, \mathsf{d})$ consisting of a set $M$ and an ultrapseudometric $\mathsf{d}$ on $M$. It is important to note that for an ultrapseudometric, we can have $\mathsf{d}(x,y)=0$ for $x \neq y$. 

\begin{definition} In a Cartesian $k$-differential category $\mathbb{X}$, for every homset $\mathbb{X}(A,B)$, define the function $\mathsf{d}_\mathsf{D}: \mathbb{X}(A,B) \times \mathbb{X}(A,B) \to \mathbb{R}_{\geq 0}$ as follows: 
\begin{align*}
\mathsf{d}_\mathsf{D}(f,g) = \begin{cases} 2^{-n} & \text{ where $n \in \mathbb{N}$ is the smallest natural number such that $\mathcal{M}^{(n)}[f] \neq \mathcal{M}^{(n)}[g]$} \\
0 & \text{ if all $n \in \mathbb{N}$, $\mathcal{M}^{(n)}[f] = \mathcal{M}^{(n)}[g]$} 
\end{cases}
\end{align*} 
We call $\mathsf{d}_\mathsf{D}(f,g)$ the \textbf{$\mathsf{D}$-distance} between $f$ and $g$. 
\end{definition}

\begin{proposition}\label{prop:ultra} For every homset of a Cartesian $k$-differential category $\mathbb{X}$, $\mathsf{d}_\mathsf{D}: \mathbb{X}(A,B) \times \mathbb{X}(A,B) \to \mathbb{R}_{\geq 0}$ is a ultrapseudometric. Moreover: 
   \begin{enumerate}[{\em (i)}]
   \item \label{prop:ultra.i} Composition is non-expansive, that is, $\mathsf{d}_\mathsf{D}\left( g_1 \circ f_1, g_2 \circ f_2 \right) \leq \mathsf{max}\lbrace \mathsf{d}_\mathsf{D}(g_1, g_2), \mathsf{d}_\mathsf{D}(f_1, f_2) \rbrace$;
   \item \label{prop:ultra.ii} The $k$-linear structure is non-expansive, that is, $\mathsf{d}_\mathsf{D}(r \cdot f_1 + s \cdot g_1, r\cdot f_2 + s\cdot g_2) \leq \mathsf{max}\lbrace \mathsf{d}_\mathsf{D}(f_1, f_2), \mathsf{d}_\mathsf{D}(g_1, g_2) \rbrace$;
   \item \label{prop:ultra.iii}  Pairing is an isometry, that is, $\mathsf{d}_\mathsf{D}\left( \langle f_0, \hdots, f_n \rangle, \langle g_0, \hdots, g_n \rangle \right) = \mathsf{max}\lbrace \mathsf{d}_\mathsf{D}(f_0, g_0), \hdots, \mathsf{d}_\mathsf{D}(f_n, g_n) \rbrace$;
   \item \label{prop:ultra.iv} Differentiation is non-expansive, that is, $\mathsf{d}_\mathsf{D}(\mathsf{D}[f],\mathsf{D}[g]) \leq \mathsf{d}_\mathsf{D}(f,g)$
\end{enumerate}
\end{proposition}
\begin{proof} We first show that $\mathsf{d}_\mathsf{D}$ is an ultrapseudometric. Trivially we have that $\mathsf{d}(x,x)=0$ and $\mathsf{d}(x,y) = \mathsf{d}(y,x)$. So we only need to check the Strong Triangle Inequality. So consider $\mathsf{d}_\mathsf{D}(f,h)$. If $\mathsf{d}_\mathsf{D}(f,h)=0$, then we are done. So suppose that $0<\mathsf{d}_\mathsf{D}(f,h)=2^{-n}$, which implies that $\mathcal{M}^{(k)}[f] = \mathcal{M}^{(k)}[g]$ for all $k < n$ and $\mathcal{M}^{(n)}[f] \neq \mathcal{M}^{(n)}[g]$. Now consider another map $g$. If $\mathsf{d}_\mathsf{D}(f,g)=0$, so $\mathcal{M}^{(j)}[f] = \mathcal{M}^{(j)}[g]$ for all $j$, it follows that $\mathsf{d}_\mathsf{D}(g,h)=2^{-n}=\mathsf{d}_\mathsf{D}(f,h)$. Similarly, if $\mathsf{d}_\mathsf{D}(g,h)=0$ then $\mathsf{d}_\mathsf{D}(f,g)=2^{-n}=\mathsf{d}_\mathsf{D}(f,h)$. We cannot have that $\mathsf{d}_\mathsf{D}(f,g)=0=\mathsf{d}_\mathsf{D}(g,h)$ since this would imply that $\mathcal{M}^{(j)}[f] = \mathcal{M}^{(j)}[g]=\mathcal{M}^{(j)}[h]$ for all $j$, which is a contradiction. Now suppose that $\mathsf{d}_\mathsf{D}(f,g)=2^{-m}$ and $\mathsf{d}(g,h)=2^{-p}$, which means that $\mathcal{M}^{(k_1)}[f] = \mathcal{M}^{(k_1)}[g]$ for all $k_1 < m$ and $\mathcal{M}^{(k_2)}[g] = \mathcal{M}^{(k_2)}[h]$ for all $k_2 < p$. If $\mathsf{max}\lbrace 2^{-m},2^{-p} \rbrace < 2^{-n}$, then $n < \mathsf{min}\lbrace m,p \rbrace$ which would imply that $\mathcal{M}^{(n)}[f] = \mathcal{M}^{(n)}[g]= \mathcal{M}^{(n)}[h]$, which is a contradiction. So we must have that $2^{-n} \leq \mathsf{max}\lbrace 2^{-m},2^{-p} \rbrace < 2^{-n}$ as desired. So $\mathsf{d}_\mathsf{D}$ is a ultrapseudometric, and thus $\left( \mathbb{X}(X,Y), \mathsf{d}_\mathsf{D} \right)$ is an ultrapseudometric space. The compatibility of the Cartesian $k$-differential structure and the ultrapseudometric is straightforward and easily follows from Lemma \ref{lemma:monomial}. 
\end{proof}

The category of ultrapseudometric spaces and non-expansive maps form a Cartesian closed category. As such, an alternative way of viewing the compatibility of the Cartesian $k$-differential structure and the ultrapseudometric (Prop \ref{prop:ultra}.(\ref{prop:ultra.i})-(\ref{prop:ultra.iv})) is saying that a Cartesian differential category is enriched over ultrapseudometric spaces via comparing Taylor differential monomials. 

Now an \textbf{ultrametric} on a set $M$ is a ultrapseudometric $\mathsf{d}: M \times M \to \mathbb{R}_{\geq 0}$ which also satisfies that: (4) $\mathsf{d}(x,y) =0$ implies $x=y$. Then an ultrametric space is a pair $(M, \mathsf{d})$ consisting of a set $M$ and an ultrametric $\mathsf{d}$ on $M$. In an arbitrary Cartesian differential category, two maps can have the same Taylor differential monomials (i.e. $\mathsf{D}$-distance of $0$) but not be equal. Therefore $\mathsf{d}_\mathsf{D}$ is usually not an ultrametric. It becomes an ultrametric when maps are completely determined by their Taylor differential monomials. When this is the case, we say that a Cartesian differential category is \emph{Taylor}. This terminology is justified since, in this Taylor setting, maps are indeed equal to their Taylor series. 

\begin{definition}  In a Cartesian $k$-differential category, we say that parallel maps $f: A \to B$ and $g: A \to B$ are \textbf{Taylor-equivalent} if for all $n \in \mathbb{N}$, $\mathcal{M}^{(n)}[f] = \mathcal{M}^{(n)}[g]$. We write $f \sim g$ if $f$ and $g$ are Taylor-equivalent. Then a Cartesian $k$-differential category $\mathbb{X}$ is said to be \textbf{Taylor} if $f \sim g$ implies $f=g$. 
\end{definition}

\begin{proposition} A Cartesian $k$-differential category $\mathbb{X}$ is Taylor if and only if for every homset, the ultrapseudometric $\mathsf{d}_\mathsf{D}: \mathbb{X}(A,B) \times \mathbb{X}(A,B) \to \mathbb{R}_{\geq 0}$ is an ultrametric.
\end{proposition}
\begin{proof} Note that by definition of $\mathsf{d}_\mathsf{D}$, we have that $\mathsf{d}_\mathsf{D}(f,g)=0$ if and only if $f \sim g$. So the desired statement follows immediately. 
\end{proof}

Of course, ultrametric spaces and non-expansive maps also form a Cartesian closed category. Thus, it follows from Prop \ref{prop:ultra}.(\ref{prop:ultra.i})-(\ref{prop:ultra.iv}) that every Taylor Cartesian differential category is enriched over the category of ultrametric spaces. We now prove the main objective of this paper: 

\begin{theorem} In a Taylor Cartesian $k$-differential category, for every map $f: A \to B$, the sequence $\left(\mathcal{T}^{(n)}[f] \right)^\infty_{n=0}$ converges to $f$ with respect to the ultrametric $\mathsf{d}_\mathsf{D}$, so we may write:
\begin{align}
    f = \sum\limits^\infty_{n=0} \mathcal{M}^{(n)}[f] && f(x) =  \sum\limits^\infty_{n=0} \dfrac{1}{n!} \cdot \dfrac{\mathsf{d}^{(n)} f(u)}{\mathsf{d}u}(0) \cdot x \cdot \hdots \cdot x
\end{align}
\end{theorem}
\begin{proof} First recall that for all $n \in \mathbb{N}$ and $k \leq n$, by Lemma \ref{lemma:monomial}.(\ref{T-1}), we have that $\mathcal{M}^{(k)}\left[ \mathcal{T}^{(n)}[f] \right] = \mathcal{M}^{(k)}[f]$. As such, the smallest possible $m$ for which $\mathcal{M}^{(m)}\left[ \mathcal{T}^{(n)}[f] \right]$ and $\mathcal{M}^{(m)}[f]$ might be different is $m=n+1$. Thus we have that $\mathsf{d}_\mathsf{D}\left( \mathcal{T}^{(n)}[f], f \right) \leq 2^{-(n+1)}$. Now let $0 < \varepsilon$ be an arbitrary non-zero positive real number. Let $N$ be the smallest natural number such that $2^{-N} \leq \varepsilon$. Then for all $n \geq N$, we have that $\mathsf{d}_\mathsf{D}\left( \mathcal{T}^{(n)}[f], f \right) \leq 2^{-(n+1)} < 2^{-n} \leq 2^{-N} \leq \varepsilon$. So we conclude that $\left(\mathcal{T}^{(n)}[f] \right)^\infty_{n=0}$ converges to $f$. 
\end{proof}

\begin{example} Let's consider our ultra(pseudo)metric in our main examples:
\begin{enumerate}[{\em (i)}]
\item In $k\text{-}\mathsf{POLY}$, the $\mathsf{D}$-distance between polynomials is given by comparing the monomial summands of same degree of the two polynomials. Since polynomials are completely determined by their monomial summands, it follows that $k\text{-}\mathsf{POLY}$ is Taylor. Moreover, we note that this ultrametric for polynomials is precisely the same as the ultrametric for power series that makes the formal infinite sum converge. 
\item In $\mathsf{SMOOTH}$, the $\mathsf{D}$-distance between smooth functions is given by comparing the usual Taylor expansions of the smooth functions. However, $\mathsf{SMOOTH}$ is not Taylor since two smooth functions can have the same Taylor expansion but not be equal. The famous counter-example is the smooth function $f:\mathbb{R} \to \mathbb{R}$ defined as $f(x) = e^{\frac{-1}{x^2}}$ for $x \neq 0$ and $f(0) =0$. The Taylor series of $f$ is $0$, so we have that $f \sim 0$. But clearly, $f \neq 0$. On the other hand, the subcategory of real \emph{entire} functions is a Taylor Cartesian differential category. 
\item The ultrapseudometric $\mathsf{d}_\mathsf{D}$ defined for the coKleisli category of a differential category with finite products is precisely the same as the ultrapseudometric defined in \cite[Sec III.E]{kerjean2023taylor}. Moreover, saying that the coKleisli category is Taylor is precisely the same as saying that the differential category is Taylor in the sense of \cite[Def III.7]{kerjean2023taylor}. Taylor differential categories were crucial for the story of developing \emph{codigging} for Differential Linear Logic in \cite{kerjean2023taylor}. Examples of Taylor differential categories can be found in \cite[Sec IV]{kerjean2023taylor}. 
\end{enumerate} 
\end{example}

While not every Cartesian differential category is Taylor, it is always possible to build one. This follows from the fact that every ultrapseudometric space can be made into an ultrametric space by quotienting. So let $(M, \mathsf{d})$ be an ultrapseudometric space. Then we have an equivalence relation $\sim$ given by $x \sim y$ if and only if $\mathsf{d}(x,y)=0$. Let $M_\sim$ be the set of equivalence classes of $\sim$, so $[x] = \lbrace y \in M \vert x \sim y \rbrace$. Then $(M_\sim, \mathsf{d}_\sim)$ is an ultrametric space where $\mathsf{d}_\sim\left( [x], [y] \right) = \mathsf{d}(x,y)$, which is indeed well-defined. Now for a Cartesian $k$-differential category $\mathbb{X}$, define the Cartesian $k$-differential category $\mathbb{X}_\sim$ as having objects the same as $\mathbb{X}$, whose homsets are $\mathbb{X}_\sim(A,B) = \mathbb{X}(A,B)_\sim$, so maps are equivalence classes $[f]: A \to B$, and where composition is defined as $[g] \circ [f] = [g \circ f]$, identity maps are $[1_A]$, the product is the same as in $\mathbb{X}$, projections are $[\pi_j]$, the pairing is $\langle [f_0], \hdots, [f_n] \rangle = \left[ \langle f_0, \hdots, f_n \rangle \right]$, the $k$-linear structure is $r \cdot [f] + s \cdot [g] = [r \cdot f + s \cdot g]$, and the differential combinator is given by $\mathsf{D}\left[ [f] \right] = \left[ \mathsf{D}[f] \right]$. 

\begin{lemma} For a Cartesian $k$-differential category $\mathbb{X}$, $\mathbb{X}_\sim$ is a Taylor Cartesian $k$-differential category.
\end{lemma}
\begin{proof} That $\mathbb{X}_\sim$ is indeed a well-defined Cartesian $k$-differential category follows from Prop \ref{prop:ultra}.(\ref{prop:ultra.i})-(\ref{prop:ultra.iv}). In fact, this is a general construction of turning an ultrapseudometric space-enriched category into an ultrametric space-enriched category. Moreover, $\mathsf{d}_\mathsf{D}$ for $\mathbb{X}_\sim$ is precisely the same as $\mathsf{d}_\sim$, which by construction is an ultrametric. So we conclude that $\mathbb{X}_\sim$ is also Taylor. 
\end{proof}

A very natural and important question to ask is: what if there was already some other established notion of convergence or infinite sum in our Cartesian differential category for which every map was equal to its Taylor series? Would the resulting Taylor series for these notions be the same as the Taylor series given by the ultrametric $\mathsf{d}_\mathsf{D}$? Under natural assumptions, the answer is yes! 

In the next section, we consider the algebraic infinite sum case, which is of high interest in more computer science-related models. We conclude this section by looking at the case of having another metric which makes the Taylor series converge. So let's consider the case when a Cartesian differential category may have other metrics on its homsets such that the sequence of Taylor differential polynomials of $f$ converges to $f$ with respect to this metric. We will show that when this is the case, our starting Cartesian differential category is Taylor, and therefore Taylor series obtained by convergences with respect to either metric is the same. So by a \textbf{metric space enriched} Cartesian $k$-differential category, we mean a Cartesian $k$-differential category $\mathbb{X}$ such that each homset $\mathbb{X}(A,B)$ comes equipped with a metric $\mathsf{b}: \mathbb{X}(A,B) \times \mathbb{X}(A,B) \to \mathbb{R}_{\geq 0}$ such that composition, the $k$-linear structure, the pairing operation, and the differential combinator are all non-expansive (so similar to the sense of Prop \ref{prop:ultra}.(\ref{prop:ultra.i})-(\ref{prop:ultra.iv})). 

\begin{lemma} In a metric space enriched Cartesian $k$-differential category $\mathbb{X}$, if for every map $f$, the sequence $\left(\mathcal{T}^{(n)}[f] \right)^\infty_{n=0}$ converges to $f$ with respect to the metric $\mathsf{b}$, then $\mathbb{X}$ is Taylor. 
\end{lemma}
\begin{proof} Suppose that $f \sim g$, so $\mathcal{M}^{(n)}[f] = \mathcal{M}^{(n)}[g]$ for all $n$, which in turn implies that $\mathcal{T}^{(n)}[f] = \mathcal{T}^{(n)}[g]$ for all $n$ as well. Consider $\mathsf{b}(f,g)$. Suppose that $\mathsf{b}(f,g) = \varepsilon \neq 0$. Now since $\left(\mathcal{T}^{(n)}[f] \right)^\infty_{n=0}$ converges to $f$ and $\left(\mathcal{T}^{(n)}[g] \right)^\infty_{n=0}$ converges to $g$ with respect to the metric $\mathsf{b}$, then there exists an $N_1$ and $N_2$ such that for all $n \geq N_1$ and $m \geq N_2$, we have that $\mathsf{b}\left(f, \mathcal{T}^{(n)}[f] \right) < \frac{\varepsilon}{2}$ and $\mathsf{b}\left(g, \mathcal{T}^{(n)}[g] \right) < \frac{\varepsilon}{2}$. Taking $N = \mathsf{max}\lbrace N_1, N_2 \rbrace$, we get that for all $k \geq N$ that $\mathsf{b}\left(f, \mathcal{T}^{(k)}[f] \right) < \frac{\varepsilon}{2}$ and $\mathsf{b}\left(g, \mathcal{T}^{(k)}[g] \right) < \frac{\varepsilon}{2}$. Then by the triangle inequality, we get that: 
\[ \varepsilon = \mathsf{b}(f,g) \leq \mathsf{b}\left(f, \mathcal{T}^{(k)}[f] \right) +\mathsf{b}\left( \mathcal{T}^{(k)}[f], g \right) = \mathsf{b}\left(f, \mathcal{T}^{(k)}[f] \right) +\mathsf{b}\left( \mathcal{T}^{(k)}[g], g \right) < \frac{\varepsilon}{2} + \frac{\varepsilon}{2} = \varepsilon  \]
    But $\varepsilon< \varepsilon$ is a contradiction, so we must have $\mathsf{b}(f,g)=0$. Since $\mathsf{b}$ is a metric, this implies that $f=g$. So we conclude that $\mathbb{X}$ is Taylor. 
\end{proof}

The converse of the above lemma is not true. Even if a Taylor Cartesian differential category has another metric on its homsets, the sequence of Taylor differential polynomials may not converge with respect to this other metric. For example, every Cartesian differential category is metric space enriched by equipping the homsets with the trivial metric, $\mathsf{b}(x,y) = 1$ if $x \neq y$ and $\mathsf{b}(x,y) = 0$ if $x=y$. However, a sequence $(x_n)^\infty_{n=0}$ converges to $x$ with respect to the trivial metric if and only if it eventually stabilizes to $x$, that is, there is some $N$ such that for all $n \geq N$, $x_n = x$. However, there are many examples of Taylor Cartesian differential categories where the sequence of Taylor differential polynomials never stabilizes.

\section{Infinite Sums and Modelling Taylor Expansion}\label{sec:infinite}

In this section, we consider Cartesian differential categories with countable infinite sums and show that being Taylor is equivalent to every map being equal to the infinite sum of its Taylor differential monomials. This implies that in a setting with countable sums, the notion of Taylor series of a map given by infinite sums or by convergence via our ultrametric is the same. From this, we then show that Manzonetto's notion of \emph{modelling Taylor expansion} in a categorical model of the differential $\lambda$-calculus with countable sums \cite{manzonetto2012categorical} is equivalent to our notion of Taylor. 

We begin by discussing Cartesian differential categories with countable sums. Briefly, recall that a commutative countably complete semiring is a commutative semiring $k$ with an appropriate notion of arbitrary countable sums, and such that these countable sums satisfy certain distributivity and partitions axioms, and that a countably complete $k$-module is a $k$-module which also has arbitrary countable sums, which again satisfy certain obvious compatibility axioms (see \cite[Chap 23]{golan2013semirings} for more details). Then we define a \textbf{left countably complete $k$-linear category} to be a category whose homsets are countably complete $k$-modules, so we can take countable infinite sums of maps, such that pre-composition preserves these countable sums: $(\sum_{i \in I} r_i \cdot f_i) \circ x = \sum_{i \in I} r_i \cdot f_i \circ x$. A map $f$ is said to be \textbf{countably $k$-linear} if post-composition by $f$ preserves the countable sums: $f \circ (\sum_{i \in I} r_i \cdot x_i) = \sum_{i \in I} r_i \cdot f \circ x_i$. Then a \textbf{Cartesian left countably complete $k$-linear category} is a left countably complete $k$-linear category with finite products such that the projections maps are countably $k$-linear. Note that a (Cartesian) left countably complete $k$-linear category is also a (Cartesian) $k$-linear category. Then the countably complete version of a Cartesian differential category is essentially the same but where we upgrade \textbf{[CD.1]} to require that the differential combinator be countably $k$-linear, and also \textbf{[CD.2]} to asking that derivatives be countably $k$-linear in their second argument. 

\begin{definition} A \textbf{Cartesian countably complete $k$-differential category} is a Cartesian left countably complete $k$-linear category equipped with a differential combinator $\mathsf{D}$ such that: 
 \begin{enumerate}[{\bf [CD.1+]}]
\item $\dfrac{\mathsf{d} \sum\limits_{i \in I} r_i \cdot f_i(x)}{\mathsf{d}x}(a) \cdot b =  \sum\limits_{i \in I} r_i \cdot \dfrac{\mathsf{d}f_i(x)}{\mathsf{d}x}(a) \cdot b$;
 \item $\dfrac{\mathsf{d}f(x)}{\mathsf{d}x}(a) \cdot (\sum\limits_{i \in I} r_i \cdot b_i) = \sum\limits_{i \in I} r_i \cdot \dfrac{\mathsf{d}f(x)}{\mathsf{d}x}(a) \cdot b_i$
 \end{enumerate}
\end{definition}

When we can scalar multiply by positive rationals, there is of course a natural notion of Taylor series of maps. In what follows, we assume that $k$ is also a commutative $\mathbb{Q}_{\geq 0}$-algebra. 

\begin{definition} In a Cartesian countably complete $k$-differential category, for a map $f: A \to B$, define the map $\mathcal{T}[f]: A \to B$ as the infinite sum of its Taylor $\mathsf{D}$-monomials:
\[\mathcal{T}[f] = \sum^\infty_{n=0} \mathcal{M}^{(n)}[f]\]
The map $\mathcal{T}[f]$ is called the \textbf{Taylor series expansion} of $f$. 
\end{definition}

\begin{proposition}\label{prop:inf1} A Cartesian countably complete $k$-differential category is Taylor if and only if for every map $f$, $\mathcal{T}[f]=f$. 
\end{proposition} 
\begin{proof} It is straightforward to see that \textbf{[CD.1+]} also implies that $\mathcal{M}^{(n)}\left[ \sum_{i \in I} r_i \cdot g_i \right] = \sum_{i \in I} r_i \cdot \mathcal{M}^{(n)}[g_i]$. From this and Lemma \ref{lemma:monomial}.(\ref{T-1}), we get that $\mathcal{M}^{(n)}\left[ \mathcal{T}[f] \right] = \mathcal{M}^{(n)}[f]$. So $\mathcal{T}[f] \sim f$. Now for the $\Rightarrow$ direction, suppose that we are in a Taylor setting. Then since $\mathcal{T}[f] \sim f$, it follows that $\mathcal{T}[f]=f$. Conversely, for the $\Leftarrow$ direction, suppose that $\mathcal{T}[f]=f$ for all maps $f$. Now consider parallel maps such that $f \sim g$, which means that $\mathcal{M}^{(n)}[f]=\mathcal{M}^{(n)}[g]$ for all $n$. This gives us that:
\[f= \mathcal{T}[f] = \sum\limits^\infty_{n=0} \mathcal{M}^{(n)}[f] = \sum\limits^\infty_{n=0} \mathcal{M}^{(n)}[g] = \mathcal{T}[g] = g\] 
So we conclude that our Cartesian countably complete $k$-differential category is Taylor. 
\end{proof}

So the above statement tells us that in a Taylor Cartesian differential category with infinite sums, every map is equal to its Taylor series whether given via infinite sums of the Taylor differential monomials or the convergence of the Taylor differential polynomials sequence with respect to the ultrametric $\mathsf{d}_\mathsf{D}$. 

We now turn our attention to Manzonetto's notion of modelling Taylor expansion in a Cartesian \emph{closed} differential category (also sometimes called a differential $\lambda$ category). As mentioned, the categorical semantics of the differential $\lambda$-calculus is provided by Cartesian \emph{closed} differential categories. Every model of the differential $\lambda$-calculus induces a Cartesian closed differential category \cite[Thm 4.3]{Cockett-2019}, and conversely, every Cartesian closed differential category induces a model of the differential $\lambda$-calculus \cite[Thm 4.12]{bucciarelli2010categorical}. For a more in-depth introduction to Cartesian closed differential categories and the differential $\lambda$-calculus, we refer the reader to \cite{bucciarelli2010categorical,manzonetto2012categorical,Cockett-2019,cockett2020linearizing, ehrhard2003differential,cockett2016categorical}. 

 In a Cartesian closed category, we denote internal homs by $[A,B]$, the evaluation maps by ${\epsilon: A \times [A.B] \to B}$, and for a map $f: C \times A \to B$, we denote its Curry by $\lambda(f): A \to [C,B]$, which recall is the unique map such that $\epsilon \circ (1_C \times \lambda(f)) = f$. In the term calculus, we write as usual: 
 \[\lambda(f)(a) = \lambda x. f(x,a)\] 
 and so $\epsilon(a, \lambda(f)(c)) = f(a,c)$, or:
 \[\left( \lambda x. f(x,a) \right)(c) = f(a,c)\]
 As explained in \cite[Lemma 4.10]{Cockett-2019}, there are two equivalent ways of expressing compatibility between the closed structure and the differential combinator: one which says that the derivative of a Curry is the Curry of the \emph{partial} derivative, and the other one which says that evaluation maps are $\mathsf{D}$-linear in the internal hom argument.
 
\begin{definition}\label{CDCcloseddef} A \textbf{Cartesian closed $k$-differential category} \cite[Def 2.9]{garner2020cartesian} is a Cartesian $k$-differential category whose underlying category is also cartesian closed, and such that one of the following equivalent axioms hold: 
\begin{description}
\item[{\bf [CD.$\lambda$]}] $\dfrac{\mathsf{d} \lambda y.f(x,y)}{\mathsf{d}x}(a) \cdot b = \lambda y. \dfrac{\mathsf{d} f(x,y)}{\mathsf{d}x}(a) \cdot b$;
\item [{\bf [CD.$\epsilon$]}] $\dfrac{\mathsf{d} \epsilon(a,j)}{\mathsf{d}j}(f) \cdot g = \epsilon(a,g)$
\end{description}
\end{definition}

From this, it follows that in a Cartesian closed $k$-differential category, the Curry operator is $k$-linear. 

\begin{lemma}\label{lem:Curry-k} In a Cartesian closed $k$-differential category, $\lambda(r \cdot f + s \cdot g) = r \cdot \lambda(f) + s \cdot \lambda(g)$, which in the term calculus is written as: 
\[  \lambda x. \left( r \cdot f(x,a) + s \cdot g(x,a) \right) =  r \cdot \lambda x.f(x,a) +  s \cdot \lambda x.g(x,a)   \]
\end{lemma} 
\begin{proof} By {\bf [CD.$\epsilon$]} and {\bf [CD.2]}, we easily compute that: 
\begin{gather*}
\epsilon\left(a, \lambda(r \cdot f + s \cdot g)(c) \right) = (r \cdot f + s \cdot g)(c,a) = r \cdot f(c,a) + s \cdot g(c,a) \\
=  r \cdot \epsilon(a, \lambda(f)(c)) +  s \cdot \epsilon(a, \lambda(g)(c)) =  r \cdot \dfrac{\mathsf{d} \epsilon(a,j)}{\mathsf{d}j}(f) \cdot \left(  \lambda( f )(c) \right)  + s \cdot \dfrac{\mathsf{d} \epsilon(a,j)}{\mathsf{d}j}(f) \cdot \left(  \lambda( g)(c) \right)\\
 = \dfrac{\mathsf{d} \epsilon(a,j)}{\mathsf{d}j}(h) \cdot \left( r \cdot \lambda(f)(c) + s\cdot \lambda (g)(c) \right) = \epsilon\left(a, \left( r \cdot\lambda( f ) + s \cdot \lambda(g) \right) (c) \right)
\end{gather*}
Since $\epsilon$ is monic in its second argument, we then get that $ \lambda(r \cdot f + s \cdot g)(c) =  \left( r \cdot\lambda( f ) + s \cdot \lambda(g) \right) (c)$. So, we conclude that the Curry operator is indeed $k$-linear as desired. 
\end{proof}

To help define Taylor expansion in a Cartesian closed differential category, Manzonetto introduces a $\star$ operation \cite[Def 4.7]{manzonetto2012categorical} which takes a map $f: C \times A \to B$ and a map $g: C \to A$, and produces a map $f \star g: C \times A \to B$ defined as:
\[(f \star g)(c,a) = \dfrac{\mathsf{d}f(c, x)}{\mathsf{d}x}(a) \cdot g(c)\] 
Successive application of the $\star$ operation is worked out to be:
\[\left(\hdots ((f \star g) \star g) \hdots) \star g\right)(c,a) = \dfrac{\mathsf{d}^{(n)}f(c, x)}{\mathsf{d}x}(a) \cdot g(c) \cdot \hdots \cdot g(c)\] 
Then in a setting with countable sums, Manzonetto defines Taylor expansion using the infinite sums of the iterated $\star$ operations. In \cite{manzonetto2012categorical}, Manzonetto worked for convenience in a setting where addition was idempotent. Here, we provide the definition in the general case where we add back the scalar multiplication in the Taylor expansion. We also give the unCurry version of the definition, which is equivalent to the Curry version. 

\begin{definition} A Cartesian closed countably complete $k$-differential category is said to \textbf{model Taylor expansion} \cite[Def 5.19]{manzonetto2012categorical} if for every map $f: C \times A \to B$ and a map $g: C \to A$,  
\begin{align}
    f(c, g(c)) = \sum\limits^\infty_{n=0} \frac{1}{n!} \cdot \frac{\mathsf{d}^{(n)}f(c, y)}{\mathsf{d}y}(0) \cdot g(c) \cdot \hdots \cdot g(c)
\end{align}
\end{definition}

We conclude with the equivalence of being Taylor and modelling Taylor expansion.  To do so, we first note that in the countably complete setting, the Curry operator is countably $k$-linear. 

\begin{lemma}\label{lem:Curry-k-inf} In a Cartesian closed countably complete $k$-differential category, the following equality holds: ${\lambda\left( \sum_{i \in I} r_i \cdot f_i \right) = \sum_{i \in I} r_i \cdot \lambda(f_i)}$, which in the term calculus is written as: 
\[  \lambda x. \left( \sum_{i \in I} r_i \cdot f_i (x,a) \right)  = \sum_{i \in I} r_i \cdot \lambda x.f_i(x,a)   \]
\end{lemma} 
\begin{proof} The proof is essentially the same as the proof of Lemma \ref{lem:Curry-k}, but where we make use of {\bf [CD.2+]} instead of simply {\bf [CD.2]}. As such, we leave this as an exercise for the reader. 
\end{proof}

\begin{proposition} A Cartesian closed countably complete $k$-differential category is Taylor if and only if it models Taylor expansion.   
\end{proposition}
\begin{proof} For the $\Rightarrow$ direction, since we are in a Taylor setting, by Prop \ref{prop:inf1} we have that $\mathcal{T}[h]=h$ for all maps $h$. Then for $f: C \times A \to B$, we compute the Taylor series of its curry $\lambda(f)$ by using Lemma \ref{lem:Curry-k-inf} to be: 
\begin{gather*}
    \mathcal{T}[\lambda(f)](a) = \sum\limits^\infty_{n=0} \frac{1}{n!} \cdot \frac{\mathsf{d}^{(n)}\lambda(f)(y)}{\mathsf{d}y}(0) \cdot a \cdot \hdots \cdot a= \sum\limits^\infty_{n=0} \frac{1}{n!} \cdot \frac{\mathsf{d}^{(n)}\lambda x.f(x,y)}{\mathsf{d}y}(0) \cdot a \cdot \hdots \cdot a \\
    = \sum\limits^\infty_{n=0}  \frac{1}{n!} \cdot \lambda x.\frac{\mathsf{d}^{(n)}f(x,y)}{\mathsf{d}y}(0) \cdot a \cdot \hdots \cdot a = \lambda x. \left( \sum\limits^\infty_{n=0}  \frac{1}{n!} \cdot \frac{\mathsf{d}^{(n)}f(x,y)}{\mathsf{d}y}(0) \cdot a \cdot \hdots \cdot a \right) 
\end{gather*}
From this, we then compute that:
\begin{gather*}
    f(c,g(c)) = \lambda(f)(g(c))(c) = \mathcal{T}[\lambda(f)](g(c))(c) \\
    = \left( \lambda x. \left( \sum\limits^\infty_{n=0}  \frac{1}{n!} \cdot \frac{\mathsf{d}^{(n)}f(x,y)}{\mathsf{d}y}(0) \cdot g(c) \cdot \hdots \cdot g(c) \right) \right)(c) =  \sum\limits^\infty_{n=0}  \frac{1}{n!} \cdot \frac{\mathsf{d}^{(n)}f(c,y)}{\mathsf{d}y}(0) \cdot g(c) \cdot \hdots \cdot g(c)
\end{gather*}
So we model Taylor expansion as desired. For the $\Leftarrow$ direction, by Prop \ref{prop:inf1}, it suffices to show that for every map $h: C \to B$, $\mathcal{T}[h]=h$. Now define $f: C \times C \to B$ as $f = h \circ \pi_2$, so $f(x,y) = h(y)$, and take $g = 1_C: C \to C$. Then using the modelling Taylor expansion formula, we compute that: 
\[h(c) = h(g(c)) = f(c,g(c))= \sum\limits^\infty_{n=0} \frac{1}{n!} \cdot \frac{\mathsf{d}^{(n)}f(c, y)}{\mathsf{d}y}(0) \cdot g(c) \cdot \hdots \cdot g(c) = \sum\limits^\infty_{n=0} \frac{1}{n!} \cdot \frac{\mathsf{d}^{(n)}h(y)}{\mathsf{d}x}(0) \cdot c \cdot \hdots \cdot c = \mathcal{T}[h](c)\] 
So $\mathcal{T}[h]=h$ as desired. 
\end{proof}

\bibliographystyle{./entics}
\bibliography{MFPS2024}

\begin{thebibliography}{10}
\providecommand{\url}[1]{\texttt{#1}}
\providecommand{\urlprefix}{ }
\providecommand{\eprint}[2][]{\url{#2}}

\bibitem{Blute2019}
Blute, R.~F., J.~R.~B. Cockett, J.-S.~P. Lemay and R.~A.~G. Seely,
  \emph{Differential categories revisited}, Applied Categorical Structures
  (2020).
\newline\urlprefix\url{https://doi.org/10.1007/s10485-019-09572-y}

\bibitem{blute2006differential}
Blute, R.~F., J.~R.~B. Cockett and R.~A.~G. Seely, \emph{Differential
  categories}, Math. Struct. Comput. Sci.  (2006).\newline\url{https://doi.org/10.1017/S0960129506005676}

\bibitem{blute2009cartesian}
Blute, R.~F., J.~R.~B. Cockett and R.~A.~G. Seely, \emph{{Cartesian
  Differential Categories}}, Theory and Applications of Categories \textbf{22},
  pages 622--672 (2009).\newline
  Available online at \url{http://www.tac.mta.ca/tac/volumes/22/23/22-23.pdf}

\bibitem{boudes2013characterization}
Boudes, P., F.~He and M.~Pagani, \emph{A characterization of the taylor
  expansion of lambda-terms}, in: \emph{Computer Science Logic 2013 (CSL
  2013)}, Schloss Dagstuhl-Leibniz-Zentrum fuer Informatik (2013).\newline\url{https://doi.org/10.4230/LIPIcs.CSL.2013.101}

\bibitem{bucciarelli2010categorical}
Bucciarelli, A., T.~Ehrhard and G.~Manzonetto, \emph{{Categorical models for
  simply typed resource calculi}}, Electronic Notes in Theoretical Computer
  Science \textbf{265}, pages 213--230 (2010).\newline\url{https://doi.org/10.1016/j.entcs.2010.08.013}

\bibitem{cockett2016categorical}
Cockett, J. and J.~Gallagher, \emph{Categorical models of the differential
  $\lambda$-calculus revisited}, Electronic Notes in Theoretical Computer
  Science \textbf{325}, pages 63--83 (2016).\newline\url{https://doi.org/10.1016/j.entcs.2016.09.032}

\bibitem{Cockett-2019}
Cockett, J. and J.~Gallagher, \emph{{Categorical models of the differential
  $\lambda$-calculus}}, Mathematical Structures in Computer Science
  \textbf{29}, pages 1513--1555 (2019).\newline\url{https://doi.org/10.1017/S0960129519000070}

\bibitem{cockett2020linearizing}
Cockett, J. R.~B. and J.-S.~P. Lemay, \emph{Linearizing combinators}, Theory
  and Applications of Categories \textbf{38}, pages 374--431 (2022).\newline Preprint available at \url{https://doi.org/10.48550/arXiv.2010.15490}

\bibitem{cockett2011faa}
Cockett, J. R.~B. and R.~A.~G. Seely, \emph{{The {Faa} di {B}runo
  construction}}, Theory and Applications of Categories \textbf{25}, pages
  394--425 (2011). \newline
  Available online at \url{http://www.tac.mta.ca/tac/volumes/25/15/25-15.pdf}

\bibitem{reversesemantics}
Cruttwell, G., J.~Gallagher and D.~Pronk, \emph{{Categorical Semantics of a
  Simple Differential Programming Language}}, Electronic Proceedings in
  Theoretical Computer Science \textbf{333}, pages 289--310 (2021).
\newline\urlprefix\url{https://doi.org/10.4204/EPTCS.333.20}

\bibitem{catML}
Cruttwell, G., B.~Gavranovic, N.~Ghani, P.~Wilson and F.~Zanasi,
  \emph{{Categorical Foundations of Gradient-Based Learning}}, ESOP 2022 pages
  1--28 (2022).
\newline\urlprefix\url{https://doi.org/10.1007/978-3-030-99336-8_1}

\bibitem{cruttwell2013forms}
Cruttwell, G. S.~H., \emph{Forms and exterior differentiation in cartesian
  differential categories}, Theory and Applications of Categories \textbf{28},
  pages 981--1001 (2013).\newline Available online at \url{http://www.tac.mta.ca/tac/volumes/28/28/28-28abs.html}

\bibitem{ehrhard2017introduction}
Ehrhard, T., \emph{An introduction to differential linear logic: proof-nets,
  models and antiderivatives}, Math. Struct. Comput. Sci.  (2017)\newline\url{https://doi.org/10.1017/S0960129516000372}

\bibitem{ehrhard2003differential}
Ehrhard, T. and L.~Regnier, \emph{{The differential lambda-calculus}},
  Theoretical Computer Science \textbf{309}, pages 1--41 (2003).\newline\url{https://doi.org/10.1016/S0304-3975(03)00392-X}

\bibitem{ehrhard2006bohm}
Ehrhard, T. and L.~Regnier, \emph{B{\"o}hm trees, krivine’s machine and the
  taylor expansion of lambda-terms}, in: \emph{Conference on Computability in
  Europe}, pages 186--197, Springer (2006).\newline Available online at \url{https://doi.org/10.1007/11780342_20}

\bibitem{ehrhard2006differential}
Ehrhard, T. and L.~Regnier, \emph{Differential interaction nets}, Theoretical
  Computer Science \textbf{364} (2006).\newline\url{https://doi.org/10.1016/j.tcs.2006.08.003}

\bibitem{ehrhard2008uniformity}
Ehrhard, T. and L.~Regnier, \emph{Uniformity and the taylor expansion of
  ordinary lambda-terms}, Theoretical Computer Science \textbf{403}, pages
  347--372 (2008).\newline\url{https://doi.org/10.1016/j.tcs.2008.06.001}

\bibitem{ehrhard2023coherent}
Ehrhard, T. and A.~Walch, \emph{Coherent taylor expansion as a bimonad}, arXiv
  preprint arXiv:2310.01907  (2023).\newline\url{https://doi.org/10.48550/arXiv.2310.01907}

\bibitem{garner2020cartesian}
Garner, R. and J.-S.~P. Lemay, \emph{{Cartesian differential categories as skew
  enriched categories}}, Applied Categorical Structures \textbf{29}, pages
  1099--1150 (2021).\newline\url{https://doi.org/10.1007/s10485-021-09649-7}

\bibitem{golan2013semirings}
Golan, J.~S., \emph{Semirings and their Applications}, Springer Science \&
  Business Media (2013). ISBN: 978-0-7923-5786-5

\bibitem{kerjean2023taylor}
Kerjean, M. and J.-S.~P. Lemay, \emph{{Taylor Expansion as a Monad in Models of
  DiLL}}, in: \emph{2023 38th Annual ACM/IEEE Symposium on Logic in Computer
  Science (LICS)}, pages 1--13, IEEE, Boston, MA, USA (2023).
\newline\urlprefix\url{https://doi.org/10.1109/LICS56636.2023.10175753}

\bibitem{lemay2020exponential}
Lemay, J.-S.~P., \emph{Exponential functions in cartesian differential
  categories}, Applied Categorical Structures \textbf{29}, pages 95--140
  (2021).\newline\url{https://doi.org/10.1007/s10485-020-09610-0}

\bibitem{EPTCS372.3}
Lemay, J.-S.~P., \emph{Jacobians and gradients for cartesian differential
  categories}, in: K.~Kishida, editor, \emph{{\rm Proceedings of the Fourth
  International Conference on} Applied Category Theory, {\rm Cambridge, United
  Kingdom, 12-16th July 2021}}, volume 372 of \emph{Electronic Proceedings in
  Theoretical Computer Science}, pages 29--42, Open Publishing Association
  (2022).
\newline\urlprefix\url{https://doi.org/10.4204/EPTCS.372.3}

\bibitem{lemay2022properties}
Lemay, J.-S.~P., \emph{{Properties and Characterisations of Cofree Cartesian
  Differential Categories}}, arXiv preprint arXiv:2210.13886  (2022).\newline\url{https://doi.org/10.48550/arXiv.2210.13886}

\bibitem{manzonetto2012categorical}
Manzonetto, G., \emph{What is a categorical model of the differential and the
  resource $\lambda$-calculi?}, Mathematical Structures in Computer Science
  \textbf{22}, pages 451--520 (2012).\newline\url{https://doi.org/10.1017/S0960129511000594}

\bibitem{manzonetto2011bohm}
Manzonetto, G. and M.~Pagani, \emph{B{\"o}hm’s theorem for resource lambda
  calculus through taylor expansion}, in: \emph{International Conference on
  Typed Lambda Calculi and Applications}, pages 153--168, Springer (2011).
  \newline\url{https://doi.org/10.1007/978-3-642-21691-6_14}

\bibitem{pagani2009inverse}
Pagani, M. and C.~Tasson, \emph{The inverse taylor expansion problem in linear
  logic}, in: \emph{2009 24th Annual IEEE Symposium on Logic In Computer
  Science}, IEEE (2009).\newline\url{https://doi.org/10.1109/LICS.2009.35}

\end{thebibliography}
\end{document}